\documentclass[a4paper,10pt]{article}

\usepackage{latexsym}
\usepackage{amsmath}
\usepackage{amsthm}
\usepackage{amssymb}
\usepackage{amscd}
\usepackage{bm}
\usepackage{graphicx}
\usepackage{wrapfig}
\usepackage{fancybox}
\usepackage{enumerate}
\usepackage{url}

\newcommand{\bbc}{{\mathbb C}}

\newcommand{\bbn}{{\mathbb N}}

\newcommand{\bbr}{{\mathbb R}}
\newcommand{\bbs}{{\mathbb S}}

\newcommand{\bbz}{{\mathbb Z}}

\newcommand{\al}{{\alpha}}

\newcommand{\gam}{{\gamma}}

\newcommand{\del}{{\delta}}
\newcommand{\Del}{{\Delta}}
\newcommand{\ep}{{\epsilon}}

\newcommand{\vph}{{\varphi}}
\newcommand{\Ome}{{\Omega}}
\newcommand{\lam}{{\lambda}}

\newcommand{\sig}{{\sigma}}

\newcommand{\im}{{\operatorname {Im}}}
\newcommand{\re}{{\operatorname {Re}}}

\newcommand{\R}{\bbr}
\newcommand{\C}{\bbc}

\newcommand{\ovl}[1]{{\overline{#1}}}

\newcommand{\rd}{{\partial}}
\newcommand{\nab}{{\nabla}}

\newcommand{\caf}{{\mathcal F}}
\newcommand{\cai}{{\mathcal{I}}}
\newcommand{\can}{{\mathcal N}}

\newcommand{\cas}{{\mathcal S}}

\newcommand{\nor}[2]{\left\| {#1} \right\|_{#2}}
\newcommand{\inp}[2]{\langle {#1} , {#2} \rangle }
\newcommand{\jb}[1]{\langle {#1} \rangle{} }
\newcommand{\dist}{\operatorname{dist}}

\newcommand{\scr}{Schr\"{o}dinger }

\newcommand{\wta}{{\widetilde{A}}}
\newcommand{\wtd}{{\widetilde{D}}}
\newcommand{\wtf}{{\widetilde{F}}}

\makeatletter

  \@addtoreset{equation}{section}
\makeatother

\newtheorem{thm}{Theorem}[section]
\newtheorem{lem}{Lemma}[section]
\newtheorem{prop}{Proposition}[section]
\newtheorem{rmk}{Remark}[section]

\renewcommand{\vec}[1]{\bar{#1}}

\begin{document}

\title{Local well-posedness for 
the Landau-Lifshitz equation with helicity term
}
\author{{\sc Ikkei Shimizu}}
\date{\hfill}
\maketitle

\begin{abstract}
We consider the initial value problem for the Landau-Lifshitz equation with helicity 
term (chiral interaction term), which arises from the Dzyaloshinskii-Moriya interaction. 
We prove that it is well-posed locally-in-time in the space $\vec k + H^s$ for $s\ge 3$ with $s\in\bbz$ and $\vec k= {}^t (0,0,1)$. We also show that 
if we further assume that the solution is homotopic to constant maps, then 
local well-posedness holds in the space $\vec k + H^s$ for $s >2$ with $s\in\R$. 
Our proof is based on the analysis via the modified Schr\"odinger map equation.
\end{abstract}


%



\section{Introduction}\label{Sc1}

We consider the initial value problem for 
the Landau-Lifshitz equation with helicity term:

\begin{equation}\label{a1}
\left\{
\begin{aligned}
\rd_t u &= u\times (-\Del u + b\,\nabla\times u)\quad \text{on } \R^2\times \R \\
u(x,0) &= u_0(x),
\end{aligned}
\right.
\end{equation}
where $u=u(x,t)$ is the unknown function from $\R^2\times \R$ to the sphere
\begin{equation*}
\bbs^2 = \left\{ y\in \R^3 : |y|=1  \right\} \subset \R^3,
\end{equation*}
$b\in \R$ is a constant, 
$\times$ denotes the vector product in $\R^3$, 
and the last term in the right hand side, called the \textit{helicity term} or \textit{chiral interaction term}, is defined by 
\begin{equation*}
\nabla \times u = 
\begin{pmatrix}
\rd_2 u_3\\
-\rd_1 u_3\\
\rd_1 u_2- \rd_2 u_1
\end{pmatrix}
.
\end{equation*}

In the physical context, 
(\ref{a1}) is considered as a mathematical model of the evolution of magnetization vectors in helimagnets, 
and the helicity term arises from the physical effect called the Dzyaloshinskii-Moriya interaction. For the detailed background of the equation, see for example 
\cite{LM,Mel} and the references therein. \bigskip\par
The equation (\ref{a1}) preserves the energy:
\begin{equation*}
\mathcal{E}(u) :=\int_{\R^2} \frac{1}{2} \left|\nabla u\right|^2 + b\, u\cdot (\nabla\times u) \, dx.
\end{equation*}
Another remarkable feature is that 
the first term of the left hand side of (\ref{a1}) does not contribute to the growth of 
$\nor{u-\vec k}{L^2}$, where $\vec k = {}^t(0,0,1)$. More precisely, we have the following identity:
\begin{equation}\label{a2}
\frac{d}{dt} \int_{\R^2} |u-\vec k|^2 dx = 2b \int_{\R^2} (u_1 \rd_1 u_3 + u_2 \rd_2 u_3) dx.
\end{equation}
In other words, the helicity term breaks the $L^2$-conservation and the spatial symmetry of the solutions. 
We note that the scaling symmetry is also destroyed by the helicity term.\bigskip\par
The corresponting energy-minimizing problem 
has been considered, for example, 
by 
\cite{DM,LM,Mel}, 
where some other additional terms are also taken into account such as easy axis anisotropy. 
The energy minimizers with nontrivial homotopy are called \textit{chiral skyrmion}, which has been attracting a lot of interest. 
On the other hand, 
the initial value problem has little been investigated so far, while D\"oring and Melcher \cite{DM} briefly mentions the local-in-time solvability in sufficient high regularities. 
The aim of the present paper is to show the local well-posedness for (\ref{a1}) in low regularities. \bigskip\par
In the case when $b=0$, in which (\ref{a1}) is especially called \textit{Schr\"odinger maps}, has been extensively studied from various aspects. 
The local well-posedness with large data is shown in \cite{M} (see also \cite{SSB}). 
For small data, the global well-posedness is proved in critical regularities by 
\cite{BIK,BIKT,IK,Sm}
(see also \cite{GKT2,S}). 
Some results on global-in-time regularity for large data can be seen in 
\cite{BIKT2,DS}. 
The asymptotic behavior of the solutions is also explored by 
\cite{BIKT,BT,GKT,GKT2,GK,GNT,MRR,P}. 
The Landau-Lifshitz equation which contains other terms 
arising from physical effects, such as easy axis or easy plane anisotropy, is also studied in \cite{dLG}, for example.\bigskip\par 
Our first main theorem is the local well-posedness for (\ref{a1}) with initial data in the Sobolev classes. We define, for function spaces $Y$, 
\begin{equation*}
\vec k + Y := \{ u: \R^2\to \bbs^2\,:\,u-\vec k \in Y \} ,\quad \vec k= {}^t (0,0,1).
\end{equation*}

\begin{thm}\label{T1}
(i) (Existence) Let $s\ge 3$ be an integer. Then, for $u_0\in \vec k+ H^s$, 
(\ref{a1}) has a (weak) solution $u\in L^\infty ([0,T]: \vec k +H^s)$, where $T=T(\nor{u_0-\vec k}{H^s},s)>0$ 
is a non-increasing function of $\nor{u_0-\vec k}{H^s}$. \\
(ii) (Uniqueness) For intervals $I\subset \R$ with $0\in I$, the solutions to (\ref{a1}) are unique in 
\begin{equation}\label{a0}
C(I:\vec k +L^2\cap L^\infty)
\cap L^\infty (I:\vec k+H^2).
\end{equation}
(iii) (Continuity) 
Let $s\ge 3$ be an integer, and let $\{ u_0^{(n)}\}_{n=1}^\infty$ be a bounded sequence in $\vec k + H^{s}$. 
Suppose that $u_0^{(n)} \to u_0$ in $\vec k+H^{s-\ep}$ for $\ep\in (0,s-1]$. 
Let $u^{(n)}, u$ be the solution in the class (\ref{a0}) with initial data $u^{(n)}_0, u_0$, respectively. 
Then, $u^{(n)} - u\to 0$ in $L^\infty_t H^{s-\ep}$.
\end{thm}


Next, we restrict our attention to the zero-homotopic solutions. 
Here we call a continuous map $f:\R^2\to \bbs^2$ \textit{zero-homotopic} if $f$ is homotopic to a constant map. 
Then we can extend the regularities of local well-posedness into $\vec k+H^s$ for $s>2$, $s\in\R$.

\begin{thm}\label{T2}
Let $s>2$ be a real number, and 
assume that $u_0\in \vec k+H^s$ is continuous and zero-homotopic. Then the followings hold. \\
(a) (Regularity) 
(\ref{a1}) has a  solution in the class 
$C_t (\vec k +H^s)$. \\
(b) (Blow-up criterion) 
Let $u\in C ([0,T_{\max}); \vec k +H^{s})$ be a solution to (\ref{a1}), where $T_{\max}$ is the maximal existence time in this class. 
Suppose $T_{\max}<\infty$. Then for any $\ep_0>0$, we have 
\begin{equation*}
\lim_{t\to T_{\max}-} \nor{u(t)-\vec k}{H^{2+\ep_0}} =\infty.
\end{equation*}
(c) (Bound) 
There exists $T=T(\nor{u_0 -\vec k}{H^s}, s)$ such that solutions $u$ to (\ref{a1}) satisfy
\begin{equation*}
\nor{u-\vec k}{L^\infty([0,T]: H^{s})} \le C(\nor{u_0 -\vec k}{H^s}, s).
\end{equation*}
(d) (Continuity) 
The solution map is continuous from $H^s$ to $C_t(\vec k+H^s)$. 
Furthermore, the following statement is true: 
Let $u^{(0)}$, $u^{(1)}$ 
be two solutions to (\ref{a1}) with 
$u^{(j)}|_{t=0} = u^{(j)}_0 \in \vec k +H^s$ for $j=0,1$. Then, 
there exists $T=T(M,s)$ such that 
\begin{equation}\label{ca1}
\nor{u^{(1)}- u^{(0)}}{L^\infty([0,T] : H^{s-1})} \le C(M ,s) \nor{u_0^{(1)} -u_0^{(0)}  }{H^{s-1}},
\end{equation}
where $M=\nor{u^{(0)}_0-\vec k}{H^s} + \nor{u^{(1)}_0-\vec k}{H^s}$.
\end{thm}

\begin{rmk}
(i) Theorems \ref{T1} and \ref{T2} also hold true for negative direction in time. 
Indeed, if we consider the transform 
\begin{equation*}
\tilde{u}(t,x_1, x_2) := 
\begin{pmatrix}
u_2(-t,x_2,x_1)\\
u_1(-t,x_2,x_1)\\
u_3(-t,x_2,x_1)
\end{pmatrix}
,
\end{equation*}
$\tilde{u}$ again satisfies (\ref{a1}) by switching the sign of $b$. \\
(ii) We note that Theorem \ref{T2} allows $s$ to be non-integer, while Theorem \ref{T1} is restricted to integer exponents. 
\end{rmk}

%
%
We briefly mention the strategy of the proof. 
For the proof of Theorem \ref{T1} (i), we employ hyperbolic-type regularization 
invented by McGahagan \cite{M}. In detail, we add second time-derivative term $\del^2 u\times \rd_{tt} u$ to the equation, and then take the limit $\del\to 0$, which yields a solution. The advantage of this argument compared with the parabolic regularization is that 
we can cancel the second derivative term $\Del u$ in the process of obtaining uniform bounds for perturbed solutions in $\del$, 
which results in the reduction of regularities in the estimate 
(see the estimate for $I_2$ in Section \ref{Sc2}, while its proof is referred to \cite{M}). 
We provide a simplified version of proof from the original argument, where the $\delta$-dependence of the maximal existence time of solutions to (\ref{b1}) is obtained by a standard scaling argument. \par
For the proof of Theorem \ref{T1} (ii), (iii), we follow the \textit{geometric} energy method 
due to McGahagan \cite{M}. More precisely, 
we introduce a geodesics connecting two solutions, and the difference of the derivatives between two solutions is measured in the same tangent space via parallel transport. 
Our argument in the present paper is rather based on the Yudovich argument \cite{Y} 
than the original paper, 
and it has already been applied in the case when $b=0$ by the author \cite{S2}.\par 
For the proof of Theorem \ref{T2}, we introduce the method using orthonormal frames, which was first applied to the study of Schr\"odinger maps in \cite{CSU}. 
The remarkable advantage is that the problem (\ref{a1}) is reduced to 
a system of semi-linear Schr\"odinger equations, called \textit{modified Schr\"odinger equation} (see (\ref{g4})). 
A remarkable difficulty here is that  in (\ref{g4}), the helicity term appears as quadratic derivative nonlinearities, which are known to be difficult to control as a perturbation of the free Schr\"odinger equation. 
To overcome that, we exploit a \textit{skew-adjoint} structure of helicity term. 
More precisely, the bad part in the nonlinearity which contains derivative losses can be absorbed into the magnetic terms, and then we can cancel them in the energy method based on the Sobolev spaces 
associated with new magnetic laplacian. 
Since the potential depends on the solutions, 
we need to estimate the commutator between time differentiation and the associated differential operators, which forces us to assume the regularity of solutions bigger than $2$. 
Our argument is reminiscent of the study of Maxwell-Schr\"odinger system by \cite{W}. \par 
We make here some remarks on the method of orthonormal frames. 
As long as the author knows, the present work is the first study of Landau-Lifshitz equation using the modified Schr\"odinger map equation in the case when 
other additional terms than Schr\"odinger map term exist in the equation. 
We note that the zero-homotopic condition is required in the construction of orthonormal frames with certain decay at spatial infinity, otherwise we have no quantitative bound for the frames. 
We also note that the modified Schr\"odinger maps naturally induces magnetic potentials, arising from the geometry of the target, and the indefiniteness of orthonormal frames provides a gauge symmetry of equations. 
We adopt here the \textit{Coulomb gauge} condition, 
which gives a simple relation between differentiated fields and connection coefficients. 
A further discussion of this gauge can be seen in Remark \ref{rm1}.
\bigskip\par
The organization of the present paper is as follows. 
Section \ref{Sc2} is devoted to the construction of solutions (Theorem \ref{T1} (i)). 
We discuss the uniqueness and the continuity of solutions with respect to the initial date in Section \ref{Sc3} (Theorem \ref{T1} (ii) and (iii)). 
In Section \ref{Sc4}, we construct the orthonormal frame and derive the properties on them. 
The nonlinear analysis for modified Schr\"odinger map equation is summarized in  Section \ref{Sc5}, which concludes Theorem \ref{T2} (a), (b) and (c). 
Section \ref{Sc6} is devoted to the proof of Theorem \ref{T2} (d).
\bigskip\par
Here is the summary of notations. 
For a Banach space $X$ and a function $f:[0,T]\to X$ with $T>0$, 
we define 
$
\nor{f}{L^p_T X} := ( \int_0^T \nor{f(t)}{X} p dt )^{1/p}
$. 
We also define $C_tX$ as the space of all continuous function from some interval $I$ to $X$, where the interval will vary in each situations. 
We often put the symbol of variables in the right bottom of the space, like $L^p_x$, in order to make clear which variable the integration is curried out in. 
For $s\in\R$, we write the maximal integer less than or equal to $s$ as $\lfloor s \rfloor$. 
We write $\rd_m$ for the differentiation with respect to $x_m$ for 
$m=1,2$. 
We shall consider the time variable as $0$-th variable, thus for example, 
$\rd_0$ stands for the differentiation with respect to $t$. 
%
We use $C$ for representing a constant, whose value varies in each situations. 
If we make clear that $C$ depends on some quantity $\sig$, 
we write it as $C_\sig$ or $C(\sig)$. 
The set of all Schwartz function from $\R^2$ to $\C$ is denoted by $\cas(\R^2)$. 
We define the Fourier transform by $\caf [f] (\xi):= \int_{\R^2} e^{-ix\cdot \xi} f(x) dx$, and the inverse Fourier transform by $\caf^{-1} [g] (x) := 
\int_{\R^2} e^{ix\cdot \xi} g(\xi) d\xi$. 
Let $|\nab|^\sig$ denote the operator defined by $|\nab|^\sig f:=  (2\pi)^{-2} \caf^{-1} |\cdot |^\sig \caf f$ for $\sig \in\R$. 
Let $H^s_r= H^s_r(\R^2)$ be the Sobolev space 
$
H^s_r := \{ f:\R^2\to \C\, |\, \nor{f}{H^s_r} := \nor{(2\pi)^{-2} \caf^{-1} (1+|\cdot|^2)^{s/2} \caf f}{L^r} <\infty \}. 
$
When $r=2$, we especially write it as $H^s$. 


%
%
%
%

\section{Existence of weak solutions}\label{Sc2}

In this section, we show the existence of weak solutions to (\ref{a1}), which is stated in Theorem \ref{T1} (i). 
We follow the hyperbolic regularizing argument of McGahagan \cite{M}, while we make some technical modifications to the original one. 
We first consider the perturbed equation:
\begin{equation}\label{b1}
-\del^2 u\times \rd_{tt} u + \rd_t u= u\times (-\Del u + b\nab\times u)
\end{equation}
with $\del>0$. (\ref{b1}) determines the evolution of maps from $\R^2$ to $\bbs^2$. 
Now we show that the initial-value problem for (\ref{b1}) is locally well-posed by using standard contraction argument. 

\begin{prop}\label{P1}
Let $s\ge 3$ be an integer, and let $(u_0, v_0)\in (\vec k + H^s) \times H^{s-1}$. 
Then there exists a unique solution $u$ with the regularity $u-\vec k\in C([0,T]: H^s)\cap C^1([0,T] : H^{s-1})$ 
for some $T=T(\nor{u_0}{H^s}, \del \nor{v_0}{H^{s-1}})$
, where $T(\cdot, \cdot)$ is a positive, non-increasing function with respect to both variables and independent of $\del$. 
\end{prop}

\begin{proof}
We change the scaling of functions as follows:
\begin{equation*}
U(x,t) := u(x, \del t).
\end{equation*}
We also write $U_0(x,t) := u_0(x, \del t)$ and $V_0:= v_0(x,\del t)$. 
Then (\ref{b1}) can be rewritten as
\begin{equation*}
\rd_{tt} U = \Del U + F_\del (U,\rd_t U),
\end{equation*}
where
\begin{equation*}
F_\del (U,V) :=  \del^{-1} U\times V + |\nab U|^2 U - |V|^2 U 
- b \sum_{j=1}^2 U_j \cdot U\times \rd_j U.
\end{equation*}
By the Duhamel formula, $(U,V)=(U,\rd_t U)$ can be written as
\begin{equation*}
\begin{aligned}
U= \Phi_U (U,V) :=&\cos (\sqrt{-\Del}t) (U_0 - \vec k) + \frac{\sin (\sqrt{-\Del} t)}{\sqrt{-\Del}} V_0 
+\vec k\\
&- \int_0^t \frac{\sin (\sqrt{-\Del} (t-\tau) )}{\sqrt{-\Del}} F(U,V) \,d\tau,\\
\end{aligned}
\end{equation*}
\begin{equation*}
\begin{aligned}
V= \Phi_V (U,V) :=&- \sqrt{-\Del} \sin (\sqrt{-\Del} t) (U_0 -\vec k) + 
\cos (\sqrt{-\Del} t)  V_0 \\
&- \int_0^t \cos (\sqrt{-\Del} (t-\tau) ) F(U,V) \,d\tau.\\
\end{aligned}
\end{equation*}
For $T,M>0$, we set the complete metric space
\begin{equation*}
\begin{aligned}
W_{T,M} := &\left\{ 
(U,V)\in C([0,T]: \vec k +H^s)\times C([0,T]: H^{s-1}) \ |\ \right. \\
&\hspace{80pt}
\left. \nor{U-U_0}{L^\infty_T H^s} + \nor{V}{L^\infty_T H^{s-1}}  \le M\right\}
\end{aligned}
\end{equation*}
with the metric
\begin{equation*}
d_{W_{T,M}} ((U,V), (\tilde{U},\tilde{V})) := \nor{U-\tilde{U}}{L^\infty_T H^{s}_x} 
+ \nor{V-\tilde{V}}{L^\infty_T H^{s-1}}.
\end{equation*}
Then by the standard argument, we can show that the map $\Phi:= (\Phi_U, \Phi_V)$ is a contraction on $W_{T,M}$ by setting
\begin{equation*}
T := \frac{C_0 \del}{\left( \nor{U_0-\vec k}{H^s} + \nor{V_0}{H^{s-1}} +1 \right)^2},\quad 
M:= C_1 \left( \nor{U_0 -\vec k}{H^s} + \nor{V_0}{H^{s-1}} \right),
\end{equation*}
for some proper choice of $C_0, C_1>0$, which concludes the existence part of the lemma. The uniqueness part can be shown in the similar way.\qed
\end{proof}

\begin{prop}\label{P2}
Let $s\ge 3$ be an integer, and let $u:\R^2\times [0,T]\to \R^3$ be a solution to (\ref{b1}), where $T$ is as in Proposition \ref{P1}. Then there exists $\tilde{T} = \tilde{T} (\nor{u_0-\vec k}{H^s}, 
\del \nor{v_0}{H^{s-1}}, \nor{v_0}{H^{s-2}})>0$ and 
$C = C (\nor{u_0-\vec k}{H^s}, 
\del \nor{v_0}{H^{s-1}}, \nor{v_0}{H^{s-2}})>0$
such that 
\begin{equation*}
\nor{\rd_t u}{L^\infty_{\tilde{T} } H^{s-2}_x} \le C. 
\end{equation*}
\end{prop}

\begin{proof}
We first prepare some notations. We set
\begin{equation*}
\mathcal{I}_k := \{ 1,2\}^k, \quad \text{and} \quad |\gam| := k \quad\text{for } \gam \in \mathcal{I}_k.
\end{equation*}
Let $t_1\in [0,T]$. 
For vector fields $X=X(x,t)\in T_{u(x,t)}\bbs^2$, 
the covariant derivative with respect to $\al$-th index is denoted by 
\begin{equation*}
D_j  X := \rd_\al X + (X \cdot \rd_\al u) u \quad \text{for } \al=0,1,2.
\end{equation*}
For $k\in\bbn\cup \{ 0\}$, we write
\begin{equation*}
D^\gam X := D_{\gam_1} \cdots D_{\gam_k} X \quad \text{for } \gam=(\gam_1,\cdots \gam_k)\in\cai_k.
\end{equation*}
Then the following inequalities hold true for $k\in\bbn$:
\begin{equation}\label{b2}
\nor{X}{H^{k}_x} \le C_k \sum_{\substack{\gam\in \cai_{k'}\\ 0\le k'\le k}} \nor{D^\gam X}{L^2_x} 
P_k \left( \sum_{\substack{\gam\in \cai_{k'}\\ 0\le k'\le k}} \nor{D^\gam X}{L^2_x} \right),
\end{equation}
\begin{equation}\label{b3}
\sum_{\substack{\gam\in \cai_{k'}\\ 0\le k'\le k}} \nor{D^\gam X}{L^2_x}  \le C_k \nor{X}{H^k_x} 
P_k (\nor{X}{H^k_x}),
\end{equation}
where $P_k$ is some polynomial depending only on $k$, and 
$C_k$ is a constant independent of $X$ (see \cite{M} for the detailed proof). 
Thus it suffices to show
\begin{equation}\label{b7}
\nor{D^\gam \rd_t u}{L^\infty_{t_1} L^2_x} \le 
C \left( 1 
+ t_1 \nor{\rd_t u}{L^\infty_{t_1} H^{s-2}_x} 
\right)
\end{equation}
for $\gam\in \cai_{k}$, $0\le k\le s-2$, 
which completes the bootstrap with respect to $\nor{\rd_t u}{L^\infty_{t_1}H^{s-2}_x}$.\bigskip\par
We set $J:= u\times \cdot$. 
Then from (\ref{b1}), we have
\begin{equation*}
\del^2 D^\gam D_t \rd_t u + J D^\gam \rd_t u = \sum_{m=1}^2 D^\gam D_m \rd_m u 
- b D^\gam (u_1 J\rd_1 u + u_2 J \rd_2 u).
\end{equation*}
If we take $L^2_x$-inner product with $JD^\gam D_t\rd_t u$, by orthogonality we have
\begin{equation*}
\begin{aligned}
\inp{D^\gam \rd_t u}{D^\gam D_t \rd_t u}_{L^2_x} 
= & \sum_{m=1}^2 \inp{D^\gam D_m \rd_m u }{ JD^\gam D_t \rd_t u}_{L^2_x} \\
&\hspace{17pt} -b \inp{D^\gam (u_1J\rd_1 u + u_2 J\rd_2 u)}{JD^\gam D_t \rd_t u}_{L^2_x},
\end{aligned}
\end{equation*}
which is equivalent to
\begin{equation*}
\frac12 \rd_t \nor{D^\gam \rd_t u}{L^2_x}^2 
= \rd_t (A_1+A_2) + \sum_{l=1}^5 I_l,
\end{equation*}
where
\begin{equation*}
A_1:= \sum_{m=1}^2 \inp{D^\gam D_m \rd_m u }{J D^\gam \rd_t u }_{L^2_x} ,
\end{equation*}
\begin{equation*}
A_2:= -b  \sum_{j=1}^2 \inp{D^\gam (u_j J\rd_j u )}{J D^\gam \rd_t u }_{L^2_x} ,
\end{equation*}
\begin{equation*}
I_1= - \inp{D^\gam \rd_t u}{[D^\gam, D_t] \rd_t u}_{L^2_x},
\end{equation*}
\begin{equation*}
I_2 = - \sum_{m=1}^2 \inp{D^\gam D_m \rd_m u}{J[D^\gam, D_t]\rd_t u}_{L^2_x},
\end{equation*}
\begin{equation*}
I_3 = - \sum_{m=1}^2 \inp{D_t D^\gam D_m \rd_m u}{JD^\gam \rd_t u}_{L^2_x},
\end{equation*}
\begin{equation*}
I_4 = -b \sum_{j=1}^2 \inp{D^\gam (u_j \rd_j u)}{J[D^\gam, D_t]\rd_t u}_{L^2_x},
\end{equation*}
\begin{equation*}
I_5 = b \sum_{j=1}^2 \inp{D_t D^\gam (u_j \rd_j u)}{JD^\gam \rd_t u}_{L^2_x}.
\end{equation*}
Thus we have
\begin{equation}\label{b4}
\nor{D^\gam \rd_t u}{L^\infty_{t_1} L^2_x}^2 \le 
\left. \nor{D^\gam \rd_t u}{L^2_x}^2 \right|_{t=0} 
+ 
\left. 2 (|A_1| +|A_2|) \right|_{t=0}^{t=t_1}
+
C \sum_{l=1}^5 \int_0^{t_1} |I_l| dt. 
\end{equation}
Let $\ep>0$ be a number determined later. 
By using the argument in \cite{M}, we have
\begin{equation*}
\left. |A_1| \right|_{t=0}^{t=t_1}
\le \ep \nor{D^\gam \rd_t u}{L^\infty_{t_1}L^2_x}^2 
+C (\ep, \nor{v_0}{H^{s-1}} , \nor{u-\vec k}{L^\infty_T H^s_x}), 
\end{equation*}
\begin{equation*}
\int_0^{t_1} ( |I_1| + |I_2| +|I_3| ) dt
\le 
C( \nor{u-\vec k}{L^\infty_{T} H^s_x})\, t_1 \nor{\rd_t u}{L^\infty_{t_1} H^{s-1}_x}^2 
P(\nor{\rd_t u}{L^\infty_{t_1} H^{s-1}_x}).
\end{equation*}
We now show that $A_2$, $I_4$, $I_5$ is also estimated as above. 
If we write
\begin{equation*}
\rd^\gam f := \rd_{\gam_1}\cdots \rd_{\gam_k} f \quad \text{for } f:\R^2\to \C,\ 
\gam= (\gam_1,\cdots ,\gam_k) \in \{ 1,2\}^k,
\end{equation*}
we have
\begin{equation*}
\begin{aligned}
\left. |A_2| \right|_{t=0}^{t=t_1}
&\le |b|
\sum_{|\gam_1|+|\gam_2|=k} \sum_{j=1,2}
\left. |\inp{\rd^{\gam_1} u_j J D^{\gam_2} \rd_j u}{JD^\gam \rd_t u}_{L^2} | \right|_{t=t_1}\\
&\hspace{90pt}+\left. |\inp{\rd^{\gam_1} u_j J D^{\gam_2} \rd_j u}{JD^\gam \rd_t u}_{L^2}| \right|_{t=0}  \\
&\le 2|b|
\sum_{|\gam_1|+|\gam_2| =k} \sum_{j=1,2} 
\nor{\rd^{\gam_1} u_j}{L^\infty_{t_1} L^\infty_x} 
\nor{D^{\gam_2} \rd_j u}{L^\infty_{t_1} L^2_x} 
\nor{D^\gam \rd_t u}{L^\infty_{t_1} L^2_x}\\
&\le \ep \nor{D^\gam \rd_t u}{L^\infty_{t_1} L^2_x} + C(\ep, \nor{u-\vec k}{L^\infty_T H^s_x})
\end{aligned}
\end{equation*}
by using the Sobolev inequality. For $I_4$, 
we use the commutator estimates obtained in \cite{M}:
\begin{equation}\label{b5}
\sum_{m=1,2}
\nor{[D^\gam, D_m] X}{L^2_x} \le C 
 \nor{u-\vec k}{H^{s}_x}^2 P(  \nor{u-\vec k}{H^{s}_x}) \nor{X}{H^{s-3}_x} ,
\end{equation}
\begin{equation}\label{b6}
\nor{[D^\gam, D_t] X}{L^2_x} \le C 
\nor{\rd_t u}{ H^{s-1}_x} \nor{u-\vec k}{H^{s}_x} P(\nor{\rd_t u}{ H^{s-1}_x} , \nor{u-\vec k}{H^{s}_x}) \nor{X}{H^{s-2}_x} . 
\end{equation}
By using (\ref{b6}), we have
\begin{equation*}
\begin{aligned}
\int_0^{t_1} |I_4| dt 
&\le |b| \int_0^{t_1} 
\sum_{\substack{|\gam_{1}| + |\gam_{2}| =k\\  }} \sum_{j=1,2}
 |\inp{\rd^{\gam_1} u_j J D^{\gam_2} \rd_j u}{J[D^\gam , D_t] \rd_t u}_{L^2} | \\
&\le 2|b| t_1
\sum_{|\gam_1|+|\gam_2| =k} \sum_{j=1,2} 
\nor{\rd^{\gam_1} u_j}{L^\infty_{t_1} L^\infty_x} 
\nor{D^{\gam_2} \rd_j u}{L^\infty_{t_1} L^2_x} 
\nor{[D^\gam , D_t] \rd_t u}{L^\infty_{t_1} L^2_x}\\
&\le C(\nor{u-\vec k}{L^\infty_T H^s_x})\, t_1\nor{\rd_t u}{L^\infty_{t_1} H^{s-1}_x}^2 
P(\nor{\rd_t u}{L^\infty_{t_1} H^{s-1}_x}).
\end{aligned}
\end{equation*}
For $I_5$, we have
\begin{equation*}
\begin{aligned}
\int_0^{t_1} |I_5| dt 
&\le |b| \sum_{j=1,2} 
\int_0^{t_1}
 |\inp{[D_t , D^{\gam}] u_j J \rd_j u}{JD^\gam \rd_t u}_{L^2} | \\
&\hspace{60pt} + |\inp{D^\gam (\rd_t u_j J \rd_j u)}{JD^\gam \rd_t u}_{L^2}|
+ |\inp{D^\gam (u_j J D_j \rd_t u)}{JD^\gam \rd_t u}_{L^2}|\, dt\\
&\le |b| \sum_{j=1,2}
\int_0^{t_1} |\inp{[D_t , D^{\gam}] u_j J \rd_j u}{JD^\gam \rd_t u}_{L^2} |\\
&\hspace{60pt}+  \sum_{|\gam_1| + |\gam_2| =k} 
|\inp{D^{\gam_1} \rd_t u_j J D^{\gam_2} \rd_j u}{JD^\gam \rd_t u}_{L^2}|\\
&\hspace{60pt}+ \sum_{\substack{|\gam_1| + |\gam_2| =k\\ |\gam_1|\neq 0 }} 
|\inp{D^{\gam_1} u_j J D^{\gam_2} D_j \rd_t u}{JD^\gam \rd_t u}_{L^2}|\\
&\hspace{60pt}+ |\inp{ u_j J [D^{\gam}, D_j] \rd_t u}{JD^\gam \rd_t u}_{L^2}|
+ |\inp{ u_j J  D_j D^{\gam} \rd_t u}{JD^\gam \rd_t u}_{L^2}| dt\\
&=: \sum_{l=1}^5 I_{5l}.
\end{aligned}
\end{equation*}
Except $I_{55}$, (\ref{b4}), (\ref{b5}) and the same argument as above yields the following estimate:
\begin{equation*}
\sum_{l=1}^4 I_{5l} \le C(\nor{u-\vec k}{L^\infty_T H^s_x})\, t_1\nor{\rd_t u}{L^\infty_{t_1} H^{s-1}_x}^2 
P(\nor{\rd_t u}{L^\infty_{t_1} H^{s-1}_x}).
\end{equation*}
For $I_{55}$, integration by part gives
\begin{equation*}
\begin{aligned}
I_{55} &= |b| \sum_{j=1,2} \int_{0}^{t_1} \left| \int_{\R^2} u_j \rd_j |D^\gam \rd_t u|^2 dx \right| dt\\
&= |b| \sum_{j=1,2} \int_0^{t_1} \left| \int_{\R^2}   \rd_j u_j  |D^\gam \rd_t u|^2 dx \right| dt\\
&\le C(\nor{u-\vec k}{L^\infty_T H^s_x})\, t_1 \nor{\rd_t u}{L^\infty_{t_1}H_x^{s-1}}^2 P(\nor{\rd_t u}{L^\infty_{t_1}H_x^{s-1}}).
\end{aligned}
\end{equation*}
Note that a sort of skew-adjoint structure of helicity term plays an essential role in the above estimate. 
Since the solution has the bound
\begin{equation*}
\nor{u-\vec k}{L^\infty_T H^{s}_x} \le C(\nor{u_0}{H^s} ,\del \nor{v_0}{H^{s-1}}), 
\end{equation*}
we obtain (\ref{b7}) as claimed.\qed
\end{proof}

\textit{Proof of Theorem 1.1 (i)}. 
We take a sequence $\{ v^{(n)}_0 \}_{n=1}^\infty \subset {H^{s-1}}$ satisfying 
\begin{equation*}
\sup_{n\in \bbn} \frac{1}{n} \nor{v_0^{(n)}}{H^{s-1}} <\infty\quad 
\text{and}
\quad
\sup_{n\in \bbn} \nor{v_0^{(n)}}{H^{s-2}} <\infty,
\end{equation*}
and let $u^{(n)}$ be a solution to (\ref{b1}) with $\del = n^{-1}$ and with 
\begin{equation*}
u^{(n)} |_{t=0} = u^{(0)},\quad \rd_t u^{(n)}|_{t=0} = v^{(n)}_0
\end{equation*}
for each $n$. Then Proposition \ref{P1} implies that the maximal existence time for $u^{(n)}$ is uniformly bounded below by some positive number $T$, and so is $\tilde{T}$ in Proposition \ref{P2}. 
Moreover, by Propositions \ref{P1} and \ref{P2}, this sequence of solutions satisfies
\begin{equation*}
\sup_{n\in\bbn} \nor{u^{(n)}}{L^\infty_T H^{s}} <\infty
\quad
\text{and}
\quad
\sup_{n\in\bbn} \nor{\rd_t u^{(n)}}{L^\infty_T H^{s-2}} <\infty.
\end{equation*}
This especially implies that 
$\{u^{(n)} \}_{n=1}^\infty$, if we take a subsequence if necessary, 
converges to some map $u$ in the sense that 
\begin{equation*}
\lim_{n\to \infty} \nor{u^{(n)}-u}{L^\infty_{\tilde{T}} L^2_x(K)} = 0 \quad \text{for every } K\Subset\R^2,
\end{equation*}
\begin{equation*}
u^{(n)} - \vec k \rightharpoonup  u -\vec k \quad \text{ in } H^s
\quad \text{ and }\quad
\rd_t u^{(n)} \to \rd_t u \quad
\text{ in } H^{s-1},
\end{equation*}
and hence $u$ satisfies (\ref{a1}) (see \cite{M} for the details). \qed 
%
%
%
%
\section{Uniqueness and Continuity}\label{Sc3}
In this section, we prove Theorem \ref{T1} (ii) and (iii). 
We mainly follow the argument of 
\cite{M,S2}. 
In this section, we shall write $\inp{f}{g} := \int_{\R^2} f\cdot \ovl{g} dx$ for $f,g:\R^2\to \C^3$. 
\bigskip\par
We first prepare the setting. 
We may only consider the positive direction in time, thus we set $I=[0,T]$ for $T>0$. 
Let $u^{(0)}$, $u^{(1)}\in C(I:L^\infty)\cap L^\infty (I:\dot{H}^1\cap\dot{H}^2)$ be two solutions. 
We may assume that $\nor{u^{(0)}|_{t=0} -u^{(1)}|_{t=0} }{L^\infty_{x}}$ is sufficiently small, and thus may choose $T$ so that $\nor{u^{(0)} -u^{(1)} }{L^\infty_{t,x}} < \pi$. \par
For $(x,t)\in \R^2\times [0,T]$, we consider a minimal geodesic from $u^{(0)}(x,t)$ to $u^{(1)}(x,t)$. 
In detail, we take a map $\gam (s,x,t): [0,1]\times \R^2\times [0,1]\to \bbs^2$ such that 
\begin{itemize}
\item $\gam (0,x,t) =u^{(0)} (x,t)$, $\gam (1,x,t) = u^{(1)} (x,t)$ in $(x,t)\in \R^2\times [0,T]$.
\item 
$\rd_{ss} \gam + (\gam_s \cdot \gam_s) \gamma = 0$ 
in $[0,1]\times\R^2\times [0,T]$.
\end{itemize}
Next, we consider the parallel transport along the geodesics. 
We define the operator $X(s,\sig): T_{\gam (\sig, x,t)} \bbs^2 \to T_{\gam (s,x,t)} \bbs^2$ such that 
for $\xi \in T_{\gam (\sig, x,t)} \bbs^2$, $X(s,\sig ) \xi$ is the parallel transport of 
$\xi$ along $\gam (s)$. 
$X(s,\sig)$ is, by definition, the resolution operator for the following ODE:
\begin{equation*}
D_s F \equiv \rd_{s} F + (F\cdot \gam_s)\gam = 0, \quad F:\R \to \R^3.
\end{equation*}

We set
\begin{equation*}
G(T) := \nor{q}{L^\infty_TL^2_x}^2 + \sum_{m=1,2}\nor{V_m}{L^\infty_TL^2_x}^2,
\end{equation*}
where
\begin{equation*}
q=u^{(0)}- u^{(1)}, \quad V_m:= X(1,0)\rd_m u^{(0)} - \rd_m u^{(1)} \text{ for } m=1,2.
\end{equation*}
In this section, we write $|u^{\max} - \vec k| := |u^{(0)} - \vec k|+ |u^{(1)} - \vec k|$ 
for simplicity, and we follow the same convention for other quantities. 
Then we claim the following:
\begin{prop}\label{P3}
Let $M = \nor{u^{\max}-\vec k}{L^\infty_T H^2_x}$. Then for all $p\in (2,\infty)$, we have
\begin{equation}\label{e1}
G(T)\le (G(0)^{\frac1p} + C(M)T)^{p}
\end{equation}
where $C(M)$ is a positive constant independent of $p$.
\end{prop}

The key estimates for the proof are the following:

\begin{prop}\label{P4}
Let $M$ as above. 
Then the following estimates are true for all $p>2$.\\
(i) $\nor{ \rd_t \nor{q }{L^2_x}^2 }{L^\infty_T} \le C(M) \nor{q}{L^\infty_T H^1_x}^{2}$.\\
(ii) $\nor{\nabla q }{L^\infty_TL^2_x} \le 
\sum_{m=1,2} \nor{V_m }{L^\infty_T L^2_x} + 
C(M)\sqrt{p} \nor{q }{L^\infty_TL^2_x}^{1-1/p}$.\\
(iii) $
\nor{\inp {\rd_t V_m}{ V_m}  }{L^\infty_T } \le C(M) p\, 
G(T)^{1-1/p}
$.
\end{prop}

We first show that these propositions imply Theorem \ref{T1} (ii) and (iii).\bigskip\par
\textit{Proof of Theorem \ref{T1} (ii).} 
We assume that $u^{(0)}|_{t=0}=u^{(1)}|_{t=0}$. 
Then we have $G(0)=0$. 
Here we may suppose that $T$ is sufficiently small such that $C(M)T<1$, 
, where $C(M)$ is as in (\ref{e1}). 
Taking the limit $p\to \infty$, we have $G(T)=0$, which completes the proof. \qed\bigskip\par
\textit{Proof of Theorem \ref{T1} (iii).} 
Let $T_{\max}^{(n)}>0$ be the maximal existence time of the solution $u^{(n)}$ in the class (\ref{a0}). 
Then Theorem \ref{T1} (i) implies that there exists $T>0$ such that 
\begin{itemize}
\item $T_{\max}^{(n)} > T$ for all $n\in\bbn$,
\item $M_2:= \sup_{n\in \bbn} \nor{u^{(n)}-\vec k}{L^\infty_T H^s} <\infty$.
\end{itemize}
Especially by the Sobolev embedding, we may take $T$ so that 
$\sup_{n\in \bbn} \nor{u^{(n)}-\vec k}{L^\infty_{t,x}([0,T]\times \R^d)} <\infty$. 
Then Proposition \ref{P3} yields
\begin{equation*}
G_{n_1,n_2} (T) \le ( G_{n_1,n_2}(0)^{1/p} + C(M_2) T)^{p},
\end{equation*}
where 
\begin{equation*}
G_{n_1,n_2} (T) := \nor{u^{(n_1)}- u^{(n_2)}}{L^\infty_TL^2_x}^2 + 
\sum_{m=1}^d\nor{ X(1,0)\rd_m u^{(n_1)} - \rd_m u^{(n_2)} }{L^\infty_TL^2_x}^2. 
\end{equation*}
By taking the limit $n_1,n_2\to \infty$, we have
\begin{equation*}
\limsup_{n_1,n_2\to \infty} G_{n_1,n_2} (T) \le (C(M_2) T)^{p}.
\end{equation*}
Here we may suppose that $T$ is sufficiently small so that $C(M_2)T<1$. 
Thus by taking $p\to \infty$, we have
\begin{equation*}
\limsup_{n_1,n_2\to \infty} G_{n_1,n_2} (T) =0,
\end{equation*}
which implies
\begin{equation*}
\lim_{n\to \infty} \nor{u^{(n)} - u}{L^\infty_T H^1} =0.
\end{equation*}
Consequently, 
by interpolation, we have
\begin{equation*}
\lim_{n\to \infty} \nor{u^{(n)} - u}{L^\infty_T H^{s-\ep}} =0.
\end{equation*}
which completes the proof.\qed\bigskip\par
We next observe that Proposition \ref{P4} implies Proposition \ref{P3}. \bigskip\par
%
\textit{Proof of Proposition \ref{P3}. } 
We first note that $V_m \in L^\infty(I:H^1)\cap W^{1,\infty} (I:\dot{H}^1)$. 
See \cite{S2} for the detailed proof. \par
Let $\ep>0$ be arbitrary number. 
Let $\psi\in C^\infty (\R)$ be a cut-off function, and define the operator
$P_k := (2\pi)^{-1} \caf^{-1} \psi (\cdot / 2^{k}) \caf $. 
Then for sufficiently large $k$ uniformly in $t$, we have
\begin{equation*}
\left| \nor{q}{L^2}^2 + \sum_{m=1,2} \inp{V_m}{P_k V_m} +\ep \right| >0 .
\end{equation*}
Since $P_k V_m \in C(I:H^1)$, we have
\begin{equation*}
\begin{aligned}
&\rd_t \left( \nor{q}{L^2}^2 + \sum_{m=1,2} \inp{V_m}{P_k V_m} +\ep \right)^{1/p} \\
=\frac{1}{p} & \left( \nor{q}{L^2}^2 + \sum_{m=1,2}\inp{V_m}{P_k V_m} +\ep \right)^{1/p-1} 
\left( \rd_t \nor{ q}{L^2_x}^2 + 2\sum_{m=1,2}\re\inp{\rd_t V_m}{P_kV_m} \right),
\end{aligned}
\end{equation*}
for all $t\in(0,T)$, and thus
\begin{equation*}
\begin{aligned}
&\left. \left( \nor{q}{L^2}^2 + \sum_{m=1,2} \inp{V_m}{P_k V_m} +\ep \right)^{1/p} \right|_{t=0}^{t=T}\\
= \int_0^T \frac{1}{p} & \left( \nor{q}{L^2}^2 + \sum_{m=1,2} \inp{V_m}{P_k V_m} +\ep \right)^{1/p-1} 
\left( \rd_t \nor{ q}{L^2_x}^2 + 2\sum_{m=1,2}\re\inp{\rd_t V_m}{P_kV_m} \right) dt.
\end{aligned}
\end{equation*}
By taking the limit $k\to\infty$, it follows that
\begin{equation*}
\begin{aligned}
& \left( G(T) +\ep \right)^{1/p} - \left( G(0) +\ep \right)^{1/p}\\
&= \int_0^T \frac{1}{p}  \left( G(T) +\ep \right)^{1/p-1} \re \left( \rd_t \nor{q}{L^2}^2 + \sum_{m=1,2} \inp{\rd_t V_m}{V_m} \right) dt\\
&\le C(M) T, 
\end{aligned}
\end{equation*}
where we used Proposition \ref{P4}. By taking $\ep\to 0$, we obtain (\ref{e1}).\hfill\qed\bigskip\par
\textit{Proof of Proposition \ref{P4}.}
For the proof of (ii), 
we refer the reader to \cite{S2}. Thus we prove the rest part.\par
(i) For a.a. $t\in I$, we have
\begin{equation}\label{by1}
\begin{aligned}
&\frac{1}{2} \frac{d}{dt} \nor{q}{L^2_x}^2
= \inp{\rd_t q}{q}_{L^2_x} \\
& = \inp{-q\times \Del u^{(0)} - u^{(1)}\times \Del q}{q}_{L^2_x}
-b \sum_{j=1}^2 \inp{q_j \rd_j u^{(0)} + u^{(1)}\rd_j q }{q}_{L^2_x} \\
& =  \sum_{j=1}^2 \inp{\rd_j (u^{(0)} + u^{(1)} ) \times \rd_j q}{q}_{L^2_x} 
+b \sum_{j=1}^2
\left(
\int_{\R^2} u^{(0)}\rd_j (|q|^2) dx -\inp{ u^{(1)}\rd_j q }{q}_{L^2_x} 
\right)
\\
&\le C (M) \nor{q}{H^1}^2.
\end{aligned}
\end{equation}

(iii) We first define $J$ by the complex structure of $\bbs^2$, which can be explicitly written as
\begin{equation*}
J\xi := p\times \xi \quad \text{ for } \xi\in T_p\bbs^2.
\end{equation*}
From (\ref{a1}), $V_m$ satisfies the following equation:
\begin{equation*}
D_t V_m = -J\sum_{k=1}^2 (D_k - \frac{b}{2} u^{(0)}_k J )^2 V_m + \sum_{\al=1}^7 R_\al,
\end{equation*}
where
\begin{equation*}
R_1= [D_t , X ]\rd_m u^{(0)},
\end{equation*}
\begin{equation*}
R_2= \sum_{k=1}^2 J [D_k , X ] D_k \rd_m u^{(0)},
\end{equation*}
\begin{equation*}
R_3= \sum_{k=1}^2 J D_k [D_k , X ] \rd_m u^{(0)},
\end{equation*}
\begin{equation*}
R_4= -J \sum_{k=1}^2 \left\{  R(X\rd_m u^{(0)}, X\rd_k u^{(0)}) X\rd_k u^{(0)} 
- R(\rd_m u^{(1)} , \rd_k u^{(1)}) \rd_k u^{(1)} \right\} ,
\end{equation*}
\begin{equation*}
R_5 = -b \sum_{k=1}^2 u^{(0)}_j [X, D_j] \rd_m u^{(0)}, 
\end{equation*}
\begin{equation*}
R_6 = -b \sum_{k=1}^2 q_j D_j \rd_m u^{(1)} ,
\end{equation*}
\begin{equation*}
R_7= \sum_{k=1}^2 
\left(
-b \rd_m u^{(0)} V_k - b\rd_m q \rd_j u^{(1)} + \frac{b}{2} \rd_k u^{(0)}_k V_m
- \frac{b^2}{4} (u^{(0)}_k)^2 J V_m
\right)
\end{equation*}
It suffices to show the following estimate for $m=1,2$ and $t\in I$:
\begin{equation}\label{ae1}
\sum_{\al=1}^7 \left| \inp{R_\al}{V_m} \right| \le C(M_1) p G(T)^{1-1/p}.
\end{equation}
For $\al=1,\cdots, 4$, we refer the reader to \cite{S2}. 
We only consider the case when $\al=5,6,7$. 
Denoting the distance between $u^{(0)}(x,t)$ and $u^{(1)}(x,t)$ on $\bbs^2$ by $l(x,t)$, we have
\begin{equation*}
\begin{aligned}
|\inp{R_5}{V_m}| 
&\le C \int_{\R^2} l |\nab u^{\max}|^2 |V_m| dx\\
&\le C 
\nor{l}{L^\infty_T L^{\frac{4p}{2p-1}}_x} 
\nor{\nab u^{\max}}{L^\infty_T L^{4p}_x}^2 
\nor{V_m}{L^\infty_T L^{\frac{4p}{2p-1}}_x}\\
&\le C p
\nor{l}{L^\infty_T L^{2}_x}^{1-1/p} \nor{l}{L^\infty_T L^{4}_x}^{1/p} 
\nor{\nab u^{\max}}{L^\infty_T H^1}^2 
\nor{V_m}{L^\infty_T L^{2}_x}^{1-1/p} \nor{V_m}{L^\infty_T L^{4}_x}^{1/p} \\
&\le C(M) p \nor{l}{L^\infty_T H^1}^{1-1/p} \nor{V_m}{L^\infty_T L^2_x}^{1-1/p},
\end{aligned}
\end{equation*}
In the third inequality above, we used the inequality
\begin{equation*}
\nor{f}{L^r} \le C\sqrt{r} \nor{f}{H^1}
\end{equation*}
for $f\in H^1$ and $r\in [2,\infty)$ (see for example \cite{O} for the proof). We also have
\begin{equation*}
\begin{aligned}
|\inp{R_6}{V_m}|
&\le C \sum_{j=1}^2 \nor{q}{L^\infty_T L^{4p}_x } \nor{D_j \rd_m u^{(0)} }{L^\infty_T L^2_x}
\nor{V_m}{L^\infty_T L^{\frac{4p}{2p-1}}_x}\\
&\le C(M) p \nor{q}{L^\infty_T H^1} \nor{V_m}{L^\infty_T L^2_x}^{1-1/p},
\end{aligned}
\end{equation*}
\begin{equation*}
\begin{aligned}
|\inp{R_7}{V_m}|
&\le C \int_{\R^2} ( |\nab u^{\max}| +1) |V_m|^2 dx + C \int_{\R^2} |\nab q| |\nab u^{\max}| |V_m| dx\\
&\le C \nor{\nab u^{\max}}{L^\infty_T L^{2p}_x} \nor{V_m}{L^\infty_T L^{\frac{4p}{2p-1}}_x}^2 + C\nor{V_m}{L^\infty_T L^2_x}^2\\
&\hspace{15pt} + C\nor{\nab q}{L^\infty_T L^{\frac{4p}{2p-1}}_x} \nor{\nab u^{\max}}{L^\infty_T L^{4p}_x}
\nor{V_m}{L^\infty_T L^{\frac{4p}{2p-1}}_x}
\\
&\le C(M) \sqrt{p} \nor{V_m}{L^\infty_T L^2_x}^{1-1/p}.
\end{aligned}
\end{equation*}
Since $|l|\le C |q|$, the proof is completed. 
\qed
%
%
%
%

\section{The method of orthonormal frames and the modified Schr\"odinger map equation}\label{Sc4}

Our proof of Theorem \ref{T2} is based on 
the method using orthonormal frames. 
In this section, we prepare the setting for the argument. \bigskip\\
Let $u=u(x,t):\R^2\times [0,T]\to \bbs^2$ be a solution in the class $C([0,T]: \vec k+ H^s)$ with $s>2$, $T>0$. We further assume that $u(\cdot, t)$ is zero-homotopic for each $t\in [0,T]$. Then, we can construct an orthonormal frame of the tangent bundle $u^{-1}T\bbs^2$ with the following property:

\begin{lem}\label{EL05}
There exist $v,w:\R^2\times [0,T]\to \R^3$ s.t. \\
(i) $v- \vec k_1$, $w- \vec k_2\in C([0,T]:H^{\lfloor s \rfloor})$ with $ \vec k_1 = {}^t(1,0,0),  \vec k_2={}^t(0,1,0)$.\\
(ii) $\{u(x,t), v(x,t),w(x,t)\}$ is a positively-oriented orthonormal basis of $\R^3$ for each $(x,t)\in\R^2\times [0,T]$.\\
(iii) $\rd_m v \cdot w \in C([0,T]:L^1)$.
\end{lem}

\begin{proof}
Our construction follows the argument in \cite{BIK}. 
We divide the proof into 2 steps.\bigskip\par
\textit{Step 1}. 
We first note that $\lim_{|x|\to \infty} u(x,t) = \vec k$ uniformly in $t$, and hense we can consider $u$ as 
a continuous map on $\bbs^2$. 
In this step, we shall construct a continuous map $v':\R^2\times [0,T] \to \bbs^2$ such that $u\cdot v'\equiv 0$, and 
$\lim_{|x|\to\infty} v'(x,t)= \vec k_1 $. 
By assumption, there is a continuous map $H:\bbs^2\times [0,1]\to \bbs^2$ such that $H(x,0)=u(x)$ and $H(x,1)=\vec k$ for $x\in \bbs^2$. 
By uniform continuity of $H$, there exists $n\in \bbn$ such that 
\begin{equation*}
\dist((x_1,t_1,\al_1),  (x_2,,t_2,\al_2)) < \frac2n\quad \Longrightarrow 
|H(x_1,t_1,\al_1) - H(x_2,t_2, \al_2)| < 2^{-10}.
\end{equation*}
Now we set
\begin{equation*}
N(U, V) := \frac{U-(U\cdot V)V}{|U-(U\cdot V)V|} \quad \text{ for } U,V\in \bbs^2. 
\end{equation*}
This is well-defined when $|U\cdot V|<2^{-5}$. 
For $(x,t,\al)\in \bbs^2 \times [0,T] \times [\frac{n-1}{n} , 1]$, we define 
\begin{equation*}
v'(x,t,\al) := N(\vec k_1 , H(x,t,\al)).
\end{equation*}
This is well-defined since $|\vec k_1\cdot H(x,t,\al)|<2^{-10}$, and satisfies $|v|\equiv 1$ and $v'\cdot H\equiv 0$ on $\bbs^2\times [0,T]\times [\frac{n-1}{n} , 1]$. 
Next, for $(x,t,\al)\in \bbs^2\times [0,T]\times [\frac{n-2}{n} , \frac{n-1}{n}]$, we define
\begin{equation*}
v'(x,t,\al) = N(v'(x,t,\frac{n-1}n) , H(x,t,\al)).
\end{equation*}
Then this is also well-defined since $|v'(x,t,\frac{n-1}{n})\cdot H(x,\al)|<2^{-10}$, and satisfies $|v'|\equiv 1$ and $v'\cdot H\equiv 0$ on $\bbs^2\times[0,T]\times [\frac{n-2}{n} , \frac{n-1}{n}]$. 
Repeating this procedure $n$ times, we obtain a map $v':\R^2\times[0,T]\times [0,1]\to \bbs^2$. 
By the above construction, $v'(x,t,0)$ is exactly the desired map in this step. \bigskip\par
\textit{Step 2}. 
We next regularize $v'$ constructed above. 
Convolution with mollifiers and multiplication by smooth cut-off function 
yield $v'':\R^2\times[0,T]\to \R^3$ with the following properties:
\begin{itemize}
\item $v'':\R^2\to \R^3$, $v''\in C^\infty$.
\item $|v''|\in [1-2^{-10}, 1+2^{-10}]$.
\item $|v''\cdot u|\le 2^{-10}$.
\item $v''(x,t) = \vec k_1$ if $|x|\ge R$ for some $R>0$ independent of $t$.
\end{itemize}
Then we define $v$ by $v= N(v'', u)$. 
Clearly we have $|v|\equiv 1$ and $v\cdot u \equiv 0$ in $\R^2\times [0,T]$. 
The task is now to check the regularity of $v$. 
Let us first prove $v\in C_t(\vec k_1 + L^2)$. 
It suffices to prove $(v-\vec k_1)\chi_{|x|>R} \in C_tL^2_x$. For $|x|>R$, we have
\begin{equation*}
v-\vec k_1 =N(\vec k_1 ,u) -\vec k_1= \frac{u_1^2}{A(1+A)} \vec k_1 - \frac{u_1}{A} u,
\quad
A= \sqrt{1-u_1^2},
\end{equation*}
which leads to the claim. 
We next check that $\rd v \in C_tL^2$ for $\rd =\rd_1,\rd_2$. It again suffices to show 
$\rd v \chi_{|x|>R} \in C_tL^2_x$. For $|x|>R$, we have
\begin{equation*}
\rd v = \frac{ (1+2A) u_1^3 \rd u_1}{A^3(1+A)^2}\vec k_1
+ \frac{2u_1\rd u_1 }{A(1+A)} \vec k_1 
- \frac{\rd u_1}{A^3}u - \frac{u_1}{A}\rd u,
\end{equation*}
which leads to the claim. The above calculation also implies that $\rd v \cdot w\in C_tL^1_x$ where $w= u\times v$. We can check the higher regularities in the similar way. 
\qed
\end{proof}
Now we next define the \textit{differentiated field} by
\begin{equation*}
\psi_m := \rd_m u \cdot v + i \rd_m \cdot w \quad \text{ for } m=0,1,2,
\end{equation*}
and the \textit{connection coefficient} by
\begin{equation*}
A_m := \rd_m v\cdot w = -v\cdot \rd_m w  \quad \text{ for } m=0,1,2.
\end{equation*}
We also set
\begin{equation*}
\psi := {}^t (\psi_1, \psi_2)\quad \text{and}\quad A:= {}^t (A_1,A_2).
\end{equation*}
Geometrically, $\psi_m$ is the representation of $\rd_m u$ in terms of the complex coordinate with axis $(v,w)$, 
and $A_m$ is the corresponding Christoffel symbol. 
%
%
These quantities are related to 
the original maps $u,v,w$ 
via the following equation:
\begin{equation}\label{Ea1}
\rd_m 
\begin{pmatrix}
u\\
v\\
w
\end{pmatrix}
= 
\begin{pmatrix}
0 & \re \psi_m & \im \psi_m \\
-\re\psi_m & 0 & A_m\\
-\im \psi_m & -A_m & 0
\end{pmatrix}
\begin{pmatrix}
u\\
v\\
w
\end{pmatrix}
.
\end{equation}
We next introduce the operator, called \textit{covariant derivative}, as follows: 
\begin{equation*}
D_m := \rd_m + iA_m \quad\text{ for } m=0,1,2.
\end{equation*}
Then $\psi$, $A$ satisfies the following relations:
\begin{equation}\label{g1}
D_m \psi_l = D_l \psi_m \quad\text{ for } l,m=0,1,2.
\end{equation}
\begin{equation}\label{g2}
[D_m, D_l] = i (\rd_m A_l - \rd_l A_m) = i\im (\psi_m \ovl{\psi_l}) \quad\text{ for } l,m=0,1,2.
\end{equation}
(\ref{g1}) represents the commutability of covariant derivatives and ordinary derivatives, 
and (\ref{g2}) represents the Ricci curvature tensor. \bigskip\par
The main advantage of our introduction of these quantities 
is the fact that $\psi$ satisfies a system of nonlinear Schr\"odigner equations with magnetic potential $A$. 
Indeed, from (\ref{a1}), we have
\begin{equation}\label{g3}
\psi_0 = -i \sum_{l=1}^2 D_l \psi_l  - b \sum_{l=1}^2 u_l \psi_l.
\end{equation}
Hence by (\ref{g1}) and (\ref{g2}), for $m=1,2$, it follows that
\begin{equation}\label{g4}
D_0 \psi_m = - i\sum_{l=1}^2 D_l D_l \psi_m + \sum_{l=1}^2 \im (\psi_m  \ovl{\psi_l}) \psi_l 
- b \sum_{l=1}^2 D_m (u_l \psi_l) . 
\end{equation}
In the case when $b=0$, 
(\ref{g4}) is called \textit{modified Schr\"odinger map equation}.
In the present paper, 
we use (\ref{g4}) only in the case when $u$ is smooth, 
and so the regularity does not 
cause any problem in the derivation of (\ref{g4}). 
However, we note that our regularity assumption 
is sufficient to obtain (\ref{g4}) rigorously by using the properties 
for $\psi$ and $A$ shown below. 
Checking this fact is left to the reader (see also \cite{BIKT2,S} for the related problem). 
\bigskip\par
Now we observe that we can retake $\{v,w\}$ such that $A$ satisfies the \textit{Coulomb gauge} condition:

\begin{prop}\label{EP1}
There exist $v,w:\R^2\to \R^3$ satisfying the following properties:\\
(i) $v-\vec k_1, w- \vec k_2 \in C_t(L^1 + L^r) \cap C_t\dot{H}^1 \cap C_t\dot{H}^{\lfloor s \rfloor}$ for all $r\in (2,\infty)$, $\vec k_1, \vec k_2 \in \bbs^2$. \\
(ii) $\{u(x,t), v(x,t),w(x,t)\}$ is a positively-oriented orthonormal basis on $\R^3$ for each $(x,t)\in \R^2\times [0,T]$.\\
(iii) The following relation holds: 
\begin{equation}\label{Ea2}
\rd_1 A_1 +\rd_2 A_2 = 0.
\end{equation}
\end{prop}

\begin{proof}
Let $\{ \tilde{v}, \tilde{w}\}$ be an orthonormal frame as in Lemma \ref{EL05}. 
Then, any other orthonormal frame $\{ v,w\}$ can be written as
\begin{equation}\label{Ea3}
\begin{pmatrix}
v\\
w
\end{pmatrix}
=
\begin{pmatrix}
\cos \chi & \sin \chi \\
-\sin \chi & \cos \chi
\end{pmatrix}
\begin{pmatrix}
\tilde{v}\\
\tilde{w}
\end{pmatrix}
\end{equation}
for $\chi : \R^2\times [0,T]\to \R$. 
It follows that $\{ v,  w\}$ satisfies (\ref{Ea2}) if and only if
\begin{equation}\label{az1}
\Del \chi = -\rd_1 \tilde{A}_1 - \rd_2 \tilde{A}_2
\end{equation}
where $\tilde{A}_m$ is the connection coefficient corresponding to $\{ \tilde{v}, \tilde{w}\}$. 
Since $\tilde{A}\in C_t(L^1\cap L^2)$, 
a solution to (\ref{az1}) can be explicitly given by
\begin{equation*}
\chi = (2\pi)^{-2} \sum_{m=1}^2 \caf^{-1} \left[ \frac{1}{|\xi|} \caf \left( R_m \tilde{A}_m \right) \right]
\in C_t(L^2 + L^{r}), \quad (2<r\le \infty),
\end{equation*}
where $R_m = (2\pi)^{-1} \caf^{-1} \frac{i\xi_m}{|\xi|} \caf$ is the Riesz operator with respect to $m$-th index. (The integrability is shown by the same argument as the proof of Lemma \ref{EL05}.)\bigskip\par
We next check the regularity for $v-\vec k_1$. (The case for $w- \vec k_2$ is similar.) 
We first write
\begin{equation*}
v-\vec k_1 = (\cos \chi -1) v + (\tilde{v}-\vec k_1) + (\sin \chi) w,
\end{equation*}
which implies $v- \vec k_1 \in C_t(L^1 + L^r)$. 
Since 
\begin{equation*}
\rd_m v = (\cos \chi) \rd_m \tilde{v} + (\sin \chi) \rd_m \tilde{w} 
- \rd_m \chi (\sin \chi) v + \rd_m \chi (\cos \chi) w,
\end{equation*}
we have $\rd_m v \in C_tL^2$. 
The higher regularities can be similarly shown.\qed
\end{proof}


\begin{prop}\label{EP2}
Under the Coulomb gauge condition,
the following properties hold true:\\
(i) $\psi \in C_tH^{s-1}$.\\
(ii) $A\in C_t(\dot{H}^1\cap \dot{H}^s)$.\\
(iii) 
$
\begin{pmatrix}
A_1\\
A_2
\end{pmatrix}
=
(2\pi)^{-2}
\caf^{-1}
\begin{pmatrix}
 |\xi|^{-1} \caf ( R_2 \im (\psi_1\ovl{\psi_2}) )  \\
 |\xi|^{-1} \caf ( R_1 \im (\psi_2\ovl{\psi_1}) ) 
\end{pmatrix}
$.
\end{prop}

\begin{proof}
(i) follows from Proposition \ref{EP1} (i) and the Leibniz rule. 
By (\ref{g2}), 
we have
\begin{equation*}
\Del A_1 = \rd_2 \im (\psi_2 \ovl{\psi_1}),
\end{equation*}
where we used (\ref{Ea2}). This leads to (iii). The case for $A_2$ is similar. 
(ii) follows from (i) and (iii).\qed
\end{proof}

We now show that $\psi$ and the gradient for $(u,v,w)$ are equivalent in the quantitative sense:

\begin{prop}\label{EP3}
Let $u,v,w$ as above. 
Then, for $\sig > 0$, there exists a polynomial $P= P_{\lfloor \sig \rfloor}$ such that the followings hold for all $t\in [0,T]$.\\
\begin{equation}\label{ay1}
\nor{\nab u }{H^\sig} + \nor{\nab v}{\dot{H}^\sig} + \nor{\nab w}{\dot{H}^\sig}
  \le \nor{\psi}{H^{\sig}} P (\nor{\psi}{H^{\sig}}) \text{ if } \sig \le 1,
\end{equation}
\begin{equation}\label{ay2}
\nor{\nab u }{H^\sig} + \nor{\nab v}{\dot{H}^\sig} + \nor{\nab w}{\dot{H}^\sig}
  \le \nor{\psi}{H^{\sig}} P (\nor{\psi}{H^{\sig-1}}) \text{ if } \sig > 1,
\end{equation}
\begin{equation}\label{ay3}
\nor{ \psi }{H^\sig}  \le \nor{\nab u}{H^{\sig}} P (\nor{\nab u}{H^{\sig}}) \text{ if } \sig \le 1,
\end{equation}
\begin{equation}\label{ay4}
\nor{ \psi }{H^\sig}  \le \nor{\nab u}{H^{\sig}} P (\nor{\nab u}{H^{\sig-1}}) \text{ if } \sig > 1.
\end{equation}
%
%
%
%
%
\end{prop}

First, we observe the following lemma.

\begin{lem}\label{EL1}
(i) For $r_1, r_2\in (2,\infty)$ with $2/r_2-1/2 \in [0,1/r_1)$, 
we have
\begin{equation*}
\nor{A}{L^{r_1}} \le C( \nor{\psi}{L^2}^2 + \nor{\psi}{L^{r_2}}^2 ) .
\end{equation*}
Especially, the above inequality holds for $(r_1,r_2)=(r,4)$ for $r\in (2,\infty)$.\\
(ii) For $\sig>0$ and $r\in (2,\infty)$ with $4/r + \sig <2$, we have 
\begin{equation*}
\nor{A}{\dot{H}^{\sig}} \le C ( \nor{\psi}{L^2}^2 + \nor{\psi}{L^{r}}^2 ) .
\end{equation*}
\end{lem}

\begin{proof}
We only see the estimate for $A_1$. 
(i) Dividing $A_1$ into low and high frequencies, we have
\begin{equation}\label{Ed1}
\begin{aligned}
\nor{\caf^{-1} ( \chi_{|\xi| \le 1} \caf A_1 )}{L^{r_1}}
&\le C
\nor{  \chi_{|\xi| \le 1} |\xi|^{-1} \caf \im (\psi_2 \ovl{\psi_1}) }{L^{r_1'}}\\
&\le C
\nor{|\xi|^{-1}\chi_{|\xi|\le 1}}{L^{r_1'}}
\nor{  \caf \im (\psi_2 \ovl{\psi_1}) }{L^{\infty}}
\le C
\nor{  \psi }{L^2}^2,
\end{aligned}
\end{equation}
\begin{equation*}
\begin{aligned}
\nor{\caf^{-1} ( \chi_{|\xi| > 1} \caf A_1 )}{L^{r_1}}
&\le C
\nor{\caf^{-1} ( |\xi|^{1-2/r_1} \chi_{|\xi| > 1} \caf A_1 )}{L^2}
\\
&\le C
\nor{  |\xi|^{-2/r_1} \chi_{|\xi| > 1} \caf \im (\psi_2 \ovl{\psi_1}) }{L^{2}}\\
&\le C \nor{|\xi|^{-2/r_1} \chi_{|\xi|>1} }{L^{\frac{2r_2}{4-r_2}}} 
\nor{\caf \im (\psi_2 \ovl{\psi_1})}{L^{\frac{r_2}{r_2-2}}}\\
&\le C
\nor{  \psi }{L^{r_2}}^2,
\end{aligned}
\end{equation*}
which completes the proof.\par
(ii) We divide $A_1$ into low and high frequencies again. Then we have
\begin{equation*}
\begin{aligned}
\nor{\caf^{-1} ( \chi_{|\xi| \le 1} \caf |\nab|^\sig A_1 )}{L^2}
&\le C
\nor{  \chi_{|\xi| \le 1} |\xi|^{-1+\sig} \caf \im (\psi_2 \ovl{\psi_1}) }{L^2}\\
&\le C
\nor{|\xi|^{-1+\sig}\chi_{|\xi|\le 1}}{L^2}
\nor{  \caf \im (\psi_2 \ovl{\psi_1}) }{L^{\infty}}
\le C
\nor{  \psi }{L^2}^2,
\end{aligned}
\end{equation*}
\begin{equation*}
\begin{aligned}
\nor{\caf^{-1} ( \chi_{|\xi| > 1} \caf |\nab|^\sig A_1 )}{L^2}
&\le C
\nor{  |\xi|^{-1+\sig} \chi_{|\xi| > 1} \caf \im (\psi_2 \ovl{\psi_1}) }{L^{2}}\\
&\le C \nor{|\xi|^{-1+\sig} \chi_{|\xi|>1} }{L^{\frac{2r}{4-r}}} 
\nor{\caf \im (\psi_2 \ovl{\psi_1})}{L^{\frac{r}{r-2}}}\\
&\le C
\nor{  \psi }{L^{r}}^2,
\end{aligned}
\end{equation*}
which completes the proof.\qed
\end{proof}

\begin{rmk}\label{rm1}
We note that $\nor{A}{L^2_x}$ cannot be controlled due to the low frequency part of  $A$. Indeed, $\nor{|\xi|^{-1}}{L^{r_1'}}$ in (\ref{Ed1}) is infinite when $r_1=2$. This kind of obstacle does not appear when the spatial dimension is higher; it is specific problem to $2D$-case. 
This is an essential obstacle of the Coulomb gauge, which causes problems 
in the analysis of Schr\"odinger maps in critical regularities 
(see \cite{BIK,BIKT} for the detailed discussions).
%
\end{rmk}

\textit{Proof of Proposition \ref{EP3}.} 
We first prove (\ref{ay1}). 
For the control of $\nor{\nab u}{H^\sig}$, we divide the case into 
$\sig\in (0,1/2]$, 
$\sig\in (1/2,1)$ and $\sig= 1$. 
When $\sig\in (0,1/2]$, 
the Leibniz rule gives 
\begin{equation}\label{Ed2}
\begin{aligned}
\nor{\nab u}{\dot{H}^\sig} &\le C( 
\nor{\psi}{\dot{H}^{\sig}} + \nor{\psi}{L^{4/(2-\sig)}} \nor{ |\nab |^\sig v}{L^{4/\sig}})\\
&\le C(
\nor{\psi}{\dot{H}^{\sig}} + \nor{\psi}{\dot{H}^{\sig /2}} \nor{ \nab v}{L^{4/(2-\sig)}})\\
&\le C
\nor{\psi}{\dot{H}^{\sig}} + C \nor{\psi}{\dot{H}^{\sig/2}} (\nor{ \psi}{L^{4/(2-\sig)}} + 
\nor{ A}{L^{4/(2-\sig)}})\\
&\le C
\nor{\psi}{\dot{H}^{\sig}} + C \nor{\psi}{\dot{H}^{\sig/2}} (\nor{ \psi}{\dot{H}^{\sig/2}} + 
\nor{ \psi}{L^2}^2 + \nor{\psi}{L^{2/(1-\sig)}}^2)\\
&\le C
\nor{\psi}{H^\sig} ( 1+ \nor{\psi}{H^\sig}^2),
\end{aligned}
\end{equation}
where the condition $\sig\le 1/2$ is used to ensure the necessary condition of Lemma \ref{EL1} (i). 
When $\sig\in (1/2, 1)$, we have
\begin{equation*}
\begin{aligned}
\nor{\nab u}{\dot{H}^\sig}
&\le C
\nor{\psi}{\dot{H}^{\sig}} + C \nor{\psi}{\dot{H}^{\sig/2}} (\nor{ \psi}{L^{4/(2-\sig)}} + 
\nor{ A}{L^{4/(2-\sig)}})\\
&\le C
\nor{\psi}{\dot{H}^{\sig}} + C \nor{\psi}{\dot{H}^{\sig/2}} (\nor{ \psi}{\dot{H}^{\sig/2}} + 
\nor{ \psi}{L^2}^2 + \nor{\psi}{L^{4}}^2)\\
&\le C
\nor{\psi}{H^\sig} ( 1+ \nor{\psi}{H^\sig}^2),
\end{aligned}
\end{equation*}
When $\sig =1$, it suffices to control
\begin{equation*}
\nor{\rd \psi \cdot v}{L^2},\quad \nor{\psi \cdot \rd v}{L^2}
\end{equation*}
for $\rd = \rd_1 ,\rd_2$. 
The former is controlled by 
$
\nor{\psi}{H^1}
$. 
The latter is estimated as follows:
\begin{equation*}
\begin{aligned}
\nor{\psi \cdot \rd v}{L^2} \le \nor{\psi}{L^4} \nor{\rd v}{L^4}
&\le C
\nor{\psi}{L^4} (\nor{\psi}{L^4} + \nor{A}{L^4} )\\
&\le C
\nor{\psi}{L^4} (\nor{\psi}{L^2}^2 + \nor{\psi}{L^4} + \nor{\psi}{L^4}^2),
\end{aligned}
\end{equation*}
where we used (\ref{Ea1}) and Lemma \ref{EL1}. 
Hence (\ref{ay1}) follows. \bigskip\par
We next prove the estimate for $v$, while that for $w$ is shown in the same way. 
It suffices to control
\begin{equation*}
\nor{|\nab|^\sig \psi \cdot u}{L^2}, \quad
\nor{\psi \cdot |\nab|^\sig u}{L^2}, \quad
\nor{|\nab|^\sig A \cdot w}{L^2}, \quad
\nor{A \cdot |\nab|^\sig  w}{L^2}.
\end{equation*}
We first see the case when $\sig\in (0,1)$. 
The first one is obviously bounded by $\nor{\psi}{\dot{H}^\sig}$. 
For the second quantity, we have
\begin{equation*}
\nor{\psi \cdot |\nab|^\sig u}{L^2} \le 
\nor{\psi}{L^{2/(1-\sig)}} \nor{|\nab|^\sig u}{L^{2/\sig}}
\le C \nor{\psi}{\dot{H}^\sig} \nor{\nab u}{L^2}
\le C \nor{\psi}{H^\sig}^2.
\end{equation*}
The third part is bounded by
\begin{equation*}
\nor{|\nab|^\sig A}{L^2}\le C( \nor{\psi}{L^2}^2 + \nor{\psi}{L^{2/(1-\sig)}}^2 )
\le C \nor{\psi}{H^{\sig}}^2,
\end{equation*}
where we used Lemma \ref{EL1} (ii). For the forth quantity, when $\sig\in (0,1/2]$, we have
\begin{equation*}
\begin{aligned}
\nor{A \cdot |\nab|^\sig  w}{L^2}
&\le C
\nor{A}{L^{4/(2-\sig)}} \nor{|\nab|^\sig w}{L^{4/\sig}}\\
&\le C
\left( \nor{\psi}{L^2}^2 + \nor{\psi}{L^{4/(2-\sig)}}^2 \right) 
\left(\nor{ \psi}{\dot{H}^{\sig/2}} + 
\nor{ \psi}{L^2}^2 + \nor{\psi}{L^{2/(2-\sig)}}^2\right)\\
&\le C \nor{\psi}{H^\sig}^3 ( 1+ \nor{\psi}{H^\sig}),
\end{aligned}
\end{equation*}
where we used the estimate for $\nor{\nab w}{L^{4/(2-\sig)}}$ in (\ref{Ed2}). 
When $\sig\in (1/2, 1)$, we have
\begin{equation*}
\begin{aligned}
\nor{A \cdot |\nab|^\sig  w}{L^2}
&\le C
\nor{A}{\dot{H}^{\sig/2}} \nor{\nab w}{L^{4/(2-\sig)}}\\
&\le C
\left( \nor{\psi}{L^2}^2 + \nor{\psi}{L^{4/(2-\sig)}}^2 \right) 
\left(\nor{ \psi}{\dot{H}^{\sig/2}} + 
\nor{ \psi}{L^2}^2 + \nor{\psi}{L^{4}}^2\right)\\
&\le C \nor{\psi}{H^\sig}^3 ( 1+ \nor{\psi}{H^\sig}).
\end{aligned}
\end{equation*}
When $\sig=1$, it suffices to control
\begin{equation*}
\nor{\rd \psi \cdot u}{L^2}, \quad
\nor{\psi \cdot \rd u}{L^2}, \quad
\nor{\rd A \cdot w}{L^2}, \quad
\nor{A \cdot \rd  w}{L^2}
\end{equation*}
for $\rd=\rd_1,\rd_2$. 
The control for the first two quantities is easy. The third term is bounded by
\begin{equation*}
\nor{\rd A}{L^2} \le C
\nor{\psi\cdot \psi}{L^2}= C \nor{\psi}{L^4}^2.
\end{equation*}
For the forth term, we have
\begin{equation*}
\nor{A \cdot \rd  w}{L^2} \le 
\nor{A}{L^4} \nor{\rd w}{L^4}
\le C \left( \nor{\psi}{L^2}^2 + \nor{\psi}{L^4}^2 \right)^2,
\end{equation*}
which leads to the conclusion. \bigskip\par
We next observe (\ref{ay2}). Let 
$
I_\sig := \nor{\nab u }{H^\sig} + \nor{\nab v}{H^\sig} + \nor{\nab w}{H^\sig}.
$ 
Then for $\sig >1$ and for $\rd=\rd_1, \rd_2$, if we take a number $\ep\in (0, \min 
\{1, \sig -1\} )$, we have
\begin{equation*}
\begin{aligned}
\nor{|\nab|^\sig ( \rd u )}{L^2}
&\le C \nor{|\nab|^{\sig-1} ( (\re\, \rd \psi) v + (\re\, \psi) \rd v + i(\re\, \rd \psi) v + i(\re\, \psi) \rd w  )}{L^2} \\
&\le C \nor{\psi}{\dot{H}^{\sig}} 
+ C \nor{\rd \psi}{L^{\frac{2}{1-\ep}}}  \nor{|\nab|^{\sig-1} v}{L^{\frac 2 \ep}}\\
&\qquad + C\nor{ |\nab|^{\sig-1} \psi}{L^{\frac{2}{\ep}}} 
\left( \nor{ \rd v}{L^{\frac{2}{1-\ep}}} 
+\nor{ \rd w}{L^{\frac{2}{1-\ep}}} 
\right)\\
&\qquad + C\nor{ \psi}{L^{\infty}} 
\left( \nor{ |\nab|^{\sig-1}  \rd v}{L^{2}}
+\nor{ |\nab|^{\sig-1}  \rd w}{L^{2}}
\right)
\\
&\le C \nor{\psi}{H^{\sig}} I_{\sig-1},
\end{aligned}
\end{equation*}
\begin{equation*}
\begin{aligned}
\nor{|\nab|^\sig ( \rd v )}{L^2}
&\le C \nor{|\nab|^{\sig-1} (  -(\re\,\rd\psi ) u - (\re\,\psi ) \rd u + (\rd A ) w + A \rd w )}{L^2} \\
&\le C \nor{\psi}{\dot{H}^{\sig}} 
+ C \nor{\rd \psi}{L^{\frac{2}{1-\ep}}}  \nor{|\nab|^{\sig-1} u}{L^{\frac 2 \ep}}\\
&\qquad + C\nor{ |\nab|^{\sig-1} \psi}{L^{\frac{2}{\ep}}}  \nor{ \rd u}{L^{\frac{2}{1-\ep}}}
+ C\nor{ \psi}{L^{\infty}}  \nor{ |\nab|^{\sig-1}  \rd u}{L^{2}}\\
& \qquad +C \nor{|\nab|^{\sig -1} \rd A}{L^2} 
+ C \nor{\rd \psi}{L^{\frac{2}{1-\ep}}}  \nor{|\nab|^{\sig-1} w}{L^{\frac 2 \ep}}\\
& \qquad + C\nor{ |\nab|^{\sig-1} A}{L^{\frac{2}{\ep}}}  \nor{ \rd w}{L^{\frac{2}{1-\ep}}}
+C\nor{A}{L^{\frac{2}{1-\ep}}}  \nor{ |\nab|^{\sig- 1} \rd w}{L^{\frac{2}{\ep}}}\\
&\le C \nor{\psi}{H^{\sig}} ( I_{\sig-1} + \nor{\psi}{H^{\sig-1}} )^2,
\end{aligned}
\end{equation*}
which yields
\begin{equation*}
I_{\sig} \le C \nor{\psi}{H^{\sig}} ( I_{\sig-1} + \nor{\psi}{H^{\sig-1}} )^2.
\end{equation*}
Hence, (\ref{ay2}) follows from (\ref{ay1}) by induction. \bigskip\par
We next prove (\ref{ay3}). 
When $\sig\in (0,1)$, the Leibniz rule yields
\begin{equation*}
\begin{aligned}
\nor{\psi_m}{\dot{H}^\sig}
&\le C
(\nor{\rd_m u}{\dot{H}^\sig} + \nor{\rd_m u}{L^{\frac{4}{2-\sig}}} \nor{|\nab|^\sig v}{L^{4/\sig}})\\
&\le C
(\nor{\rd_m u}{\dot{H}^\sig} + \nor{\rd_m u}{\dot{H}^{2/\sig}} \nor{\nab v}{L^{4/(2-\sig)}})\\
&\le C
\nor{\rd_m u}{\dot{H}^\sig} (1 + \nor{\psi}{L^{4/(2-\sig)}} +\nor{A}{L^{4/(2-\sig)}})\\
&\le C
\nor{\rd_m u}{\dot{H}^\sig} (1 +\nor{\psi}{L^2}^2 + \nor{\psi}{L^{4/(2-\sig)}}^2)\\
&\le C
\nor{\rd_m u}{\dot{H}^\sig} (1 + \nor{\nab u}{H^{\sig/2}}^2), 
\end{aligned}
\end{equation*}
where we used (\ref{Ea1}) and Lemma \ref{EL1}. When $\sig=1$, it suffices to control 
\begin{equation*}
\nor{\rd ^2 u \cdot v}{L^2},\quad \nor{\rd u \cdot \rd v}{L^2}
\end{equation*}
for $\rd=\rd_1,\rd_2$. 
The former is controlled by $\nor{\nab u}{H^1}$. 
The latter is estimated as follows:
\begin{equation*}
\begin{aligned}
\nor{\rd u \cdot \rd v}{L^2}
\le 
\nor{\rd u}{L^4} \nor{\rd v}{L^4} 
&\le C \nor{\nab u}{H^1} (1+ \nor{\psi}{L^4} + \nor{A}{L^4})\\
&\le C \nor{\nab u}{H^1} (1+ \nor{\nab u}{H^1}^2),
\end{aligned}
\end{equation*}
which leads to (\ref{ay3}). \bigskip\par
It remains to prove (\ref{ay4}). 
For $\sig >1$, $m=1,2$ and for $\rd=\rd_1, \rd_2$, if we take a number $\ep\in (0, \min 
\{1, \sig -1\} )$, we have
\begin{equation*}
\begin{aligned}
\nor{|\nab |^{\sig-1} \rd \psi_m}{L^2}
&= \nor{|\nab |^{\sig-1} (\rd \rd_m u \cdot v + \rd_m u \cdot \rd v + i\rd \rd_m u \cdot w + i\rd_m u \cdot \rd w)}{L^2} \\
&\le C \nor{\nab u}{H^\sig} 
+ C\nor{\rd \rd_m u}{L^{\frac{2}{1-\ep}}}  \left( \nor{ |\nab|^{\sig -1} v}{L^{\frac{2}{\ep}}}  + \nor{ |\nab|^{\sig -1} v}{L^{\frac{2}{\ep}}} \right)  \\
&\qquad+ C\nor{|\nab |^{\sig -1} \rd_m u}{L^{\frac{2}{\ep}}} \left( \nor{ \rd v}{L^{\frac{2}{1-\ep}}}
 +\nor{ \rd v}{L^{\frac{2}{1-\ep}}} \right) \\
&\qquad+ C\nor{\rd_m u}{L^\infty} \left( \nor{ |\nab|^{\sig-1} \rd v}{L^{2}}
 +\nor{ |\nab|^{\sig-1} \rd w}{L^{2}}\right) \\
&\le C \nor{ \nab u }{H^\sig} \left( \nor{\nab v}{H^{\sig-1}} + \nor{\nab w}{H^{\sig -1}} \right) \\
&\le \nor{\nab u}{H^\sig} P_{\lfloor \sig -1 \rfloor} (\nor{\psi}{H^{\sig-1}}),
\end{aligned}
\end{equation*}
where we used (\ref{ay1}) and (\ref{ay2}). 
Thus (\ref{ay4}) follows by induction.\qed\bigskip\par
%
%
%
%
Finally we observe the properties concerning time derivatives.

\begin{prop}\label{EP4}
Let $u,v,w$ as above. 
Then, for $\sig >0$, there exists a polynomial $P= P_{\lfloor \sig \rfloor}$ such that the followings are true for all $t\in [0,T]$.\\
(i) 
$
\nor{\psi_0}{{H}^{\sig}}\le \nor{\psi}{H^{1+\sig}} P(\nor{\psi}{H^{\sig}})
$.\\
(ii)
$
A_0 = (2\pi)^{-2} \sum_{j=1}^2 \mathcal{F}^{-1} \left[ |\xi|^{-1} \caf R_j \im (\psi_0 \ovl{\psi_j}) \right]
$.\\
(iii) 
$
\nor{\nab A_0}{{H}^{\sig}} \le C \nor{\psi}{H^{\sig+1}}  P(\nor{\psi}{H^{\sig}})
$.
\end{prop}

\begin{proof}
By using (\ref{g3}), (i) follows by the same argument as in the proof of Proposition \ref{EP3}.
(ii) follows from (\ref{g2}). By using (ii), we have
\begin{equation*}
\begin{aligned}
\nor{\nab A_0}{H^\sig} \le C \sum_{j=1}^d \nor{\psi_j \psi_0}{H^\sig}
&\le C ( \nor{\psi}{H^\sig} \nor{\psi_0}{L^2} + \nor{\psi}{L^2} \nor{\psi_0}{H^\sig})\\
&\le \nor{\psi}{H^{1+\sig}}  P(\nor{\psi}{H^{\sig}}),
\end{aligned}
\end{equation*}
which gives (iii).\qed
\end{proof}

%
%
%
%

\section{A priori estimate}\label{Sc5}

In this section, we prove Theorem \ref{T2} (a), (b) and (c). 
The key estimate is the following:
\begin{prop}\label{Pa1}
Let $L\ge 1$ be an integer, and let $\ep$ be a number in $[0,1)$ when $L\ge 2$, or in $(0,1)$ when $L=1$. 
When $\ep>0$, we also set $\del \in (0,\ep)$. 
Let $u\in C([0,T], \vec k + H^{L+\ep})$ be a solution to (\ref{a1}), and set 
$Y_T := \sum_{|\al | = L, m=1,2} \nor{\rd^\al_x \psi_m}{L^\infty_T H^\ep}$. 
We further set $M_T := \nor{\psi}{L^\infty_T H^{L-\frac{1}{2}}}$ when $L\ge 2$, or 
$M_T := \nor{\psi}{L^\infty_T H^{1+\del}}$ when $L=1$. 
Then there exists $K>0$, independent of $L$, $\ep$, and $\del$, 
such that the following inequality holds true.\\
\begin{equation}\label{aa1}
Y_T \le C \nor{u_0-\vec k}{H^{L+1+\ep}} + T(1+T)^K Y_T P_L (M_T), 
\end{equation}
where $C=C(\nor{u_0 - \vec k}{H^{\max \{L+\ep ,\frac 3 2 \} }}, L) >0$ is a constant, and 
$P_L$ is a polynomial.
\end{prop}

We first show that Proposition \ref{Pa1} implies the conclusions. \bigskip\par

\textit{Proof of Theorem \ref{T2} (a), (b) and (c).} 
We divide the proof into 3 steps:\bigskip\par
\textit{Step 1}. 
We suppose here that $u_0\in \vec k +H^\infty$. 
Theorem \ref{T1} implies that for $L\ge 3$ with $L\in\bbz$, there exists 
a unique solution to (\ref{a1}) $u\in L^\infty ([0,T_{\max}^L), \vec k + H^L)$ where $T_{\max}^L(u_0)$ is the maximal existence time in this class. 
Now we consider the following proposition for $s,s'\in\R$:\bigskip\par
$(P)_{s, s'}$: $T_{\max}^s = T_{\max}^{s'}$.\bigskip\par
In this step, we shall show that $(P)_{s,3}$ is true for $s\ge 3, s\in \mathbb{Z}$ by induction with respect to $s$. 
We assume that for some $s\ge 3, s\in \mathbb{Z}$, $(P)_{s',3}$ is true for all $s'\in [3,s]\cap \mathbb{Z}$. 
Suppose $T_{\max}^{s+1}< T_{\max}^{3}$ and we will show the contradiction. \par
We apply Proposition \ref{Pa1} with $L={s-1}$. Then Proposition \ref{EP3} implies
\begin{equation*}
\sum_{|\al | = s, m=1,2} \nor{\rd^\al_x \psi_m}{L^\infty_T H^\ep}
\le C(\nor{u_0-\vec k}{H^{s-1+\ep}}) \nor{u_0-\vec k}{H^{s+\ep}}
\end{equation*}
for all $T<  T_{\max}^{s+1} \wedge T(\nor{u-\vec k}{L^\infty_{T^{s+1}_{\max} }H^s})  $. 
On the other hand, (\ref{a2}) yields
\begin{equation*}
\begin{aligned}
\nor{u-\vec k}{L^\infty_T L^2_x}^2 &\le \nor{u_0- \vec k}{L^2}^2 + C T \nor{u-\vec k}{L^\infty_T L^2_x} \nor{\nab u}{L^\infty_T L^2_x}\\
&\le \nor{u_0- \vec k}{L^2}^2 + \frac 1 2 \nor{u-\vec k}{L^\infty_T L^2_x}^2 +CT^2 \nor{\nab u}{L^\infty_T L^2_x}^2,
\end{aligned}
\end{equation*}
which implies
\begin{equation}\label{aa15}
\nor{u-\vec k}{L^\infty_T L^2_x} \le C \nor{u_0- \vec k}{L^2} + C T \nor{\nab u}{L^\infty_T L^2_x}.
\end{equation}
Thus it follows that
\begin{equation}\label{aa2}
\nor{u-\vec k}{L^\infty_T H^{s+\ep}}
\le C(\nor{u_0-\vec k}{H^{s-1 +\ep}} ) \nor{u_0-\vec k}{H^{s+\ep}}.
\end{equation}
for $T<  T_{\max}^{s+1} \wedge T(\nor{u-\vec k}{L^\infty_{T^{s+1}_{\max} }H^s})  $. 
Repeating this procedure in finite times yields the same bound 
for $T= T_{\max}^{s+1} $ with modifying the constant in the right hand side. \par
Next we apply Proposition \ref{Pa1} for $L=s$, $\ep=0$. Then, the same argument above yields
\begin{equation}\label{aa3}
\nor{u-\vec k}{L^\infty_T H^{s+1}}
\le C(\nor{u_0-\vec k}{H^{s}} ) \nor{u_0-\vec k}{ H^{s+1}}
\end{equation}
for all $T< T_{\max}^{s+1} \wedge T(\nor{u-\vec k}{L^\infty_{T^{s+1}_{\max}} H^s})$, where we used (\ref{aa2}). If we use (\ref{aa3}) repeatedly, we have 
the same bound for $T=T_{\max}^{s+1}$ with a certain change of constant, 
which contradicts to Theorem \ref{T1}. \bigskip\par 
As a consequence, for $u_0\in \vec k+H^\infty$ we have a unique solution $u\in C ([0,T^\infty_{\max}) : \vec k + H^\infty)$, where $T^\infty_{\max}$ is the maximal existence time in this class, and we have $T^\infty_{\max} = T^3_{\max}$.
\bigskip\par

\textit{Step 2}. 
Next, we prove Theorem \ref{T2} (a). 
For $u_0\in \vec k +H^{s}$ with $s>2$, $s\in\R\backslash \bbz$, we shall construct a unique solution $u\in L^\infty_T(\vec k +H^s)$ with some $T=T(\nor{u_0 - \vec k}{H^s})$. By standard compactness argument, it suffices to show
\begin{equation}\label{aa35}
\nor{u-\vec k}{L^\infty_T H^s} \le C (\nor{u_0 - \vec k}{H^s} )
\end{equation}
for $u_0 \in \vec k + H^\infty$ with some $T=T(\nor{u_0- \vec k}{H^s})$.\par 
By Propositions \ref{Pa1} and \ref{EP3}, and by (\ref{aa15}) we have
\begin{equation}\label{aa4}
\nor{u-\vec k}{L^\infty_T H^{s}} \le C + C T (1+T)^K \nor{u-\vec k}{L^\infty_T H^{s}} P_s 
(\nor{u-\vec k}{L^\infty_T H^{s}})
\end{equation}
for $T<T^\infty_{\max}$ with $C=C(\nor{u_0-\vec k}{H^s})$. 
By bootstrapping argument, (\ref{aa4}) yields (\ref{aa35}) 
for all $T<T(\nor{u_0-\vec k}{H^s})\wedge T^\infty_{\max} \wedge 1$. 
Then the same argument as in Step 1 shows 
$T^\infty_{\max} \ge T(\nor{u_0-\vec k}{H^s}) \wedge 1$, which concludes the claim. \par
The fact that the solution above is in $C_t(\vec k + H^s)$ can be shown by 
using Theorem \ref{T2} (d), which will be proved in Section \ref{Sc6} independently of this fact.
\bigskip

\textit{Step 3}. 
We finally show the rest part. 
(c) follows from (\ref{aa35}) via standard approximating argument. 
In particular, this bound implies that if the maximal existence time $T^s_{\max}$ in the class $L^\infty_T H^s$ is finite, then
\begin{equation}
\lim_{t\to T^s_{\max}-} \nor{u(t)- \vec k}{H^s} = \infty.
\end{equation} 
We can then show (b) by the same argument as in Step 1.\qed\bigskip\par

The idea to prove Proposition \ref{Pa1} is as follows. We can rewrite the equation (\ref{g4}) as
\begin{equation}\label{ab1}
D_0 \psi_m = -i \sum_{l=1}^2 \widetilde{D}_l \widetilde{D}_l \psi_m + \can_m
\end{equation}
for $m=1,2$, where
\begin{equation*}
\wtd_l := \rd_l + i \wta_l,\quad \wta_l := A_l  - \frac{1}{2} b u_l \quad (l=1,2),
\end{equation*}
\begin{equation*}
\can_m := b\sum_{l=1}^2 
(\frac{1}{2} (\rd_l u_l)\psi_m - \frac{ib}{4}u_l^2 \psi_m - (\rd_m u_l) \psi_l ) +\sum_{l=1}^2 \Im (\psi_m \ovl{\psi_l}) \psi_l\quad (m=1,2).
\end{equation*}

Note that $i\sum_{l=1}^2 \wtd_l \wtd_l$ is a skew-symmetric operator on $L^2(\R^2)$, and that $\can_m$ includes no derivatives of $\psi_m$. 
This structure allows us to apply the energy method by introducing the differential operator associated with magnetic potential $\wta = {}^t (\wta_1 ,\wta_2)$. \bigskip\par

Writing $\Del_\wta := \sum_{l=1}^2 \wtd_l \wtd_l$, 
we have the following lemma:
\begin{lem}\label{La1}
Let $X= L^4 (\R^2)\cap \dot{H}^1 (\R^2)$. 
Then the followings are true.
\\
(i) 
$1-\Del_\wta$ is a self-adjoint, positive operator on $L^2(\R^2)$ for each $t\in [0,T]$ with domain $H^2(\R^2)$.\\
(ii) Let $\Ome_\wta := (1-\Del_\wta)^{1/2}$. Then there exist constants 
$C,K>0$ with $K\in \bbz$ such that for all $s\in [-2,2]$ and $t\in [0,T]$, we have
\begin{equation*}
C^{-1} \jb{\nor{\wta}{X}}^{-K} \nor{f}{H^s} \le \nor{\Ome_\wta^s f }{L^2} 
\le C \jb{\nor{\wta}{X}}^K \nor{f}{H^s} .
\end{equation*}
(iii) For all $s\in [-2,2]$ and $\lam > 0$, we have
\begin{equation*}
\nor{\Ome_\wta^s (\Ome_\wta^2 +\lam)^{-1} f}{L^2} \le C \jb{\lam}^{-1+\frac{s}{2}} \nor{f}{L^2}.
\end{equation*}
(iv) Let $\ep>0$ and $r\in (2,\infty)$. Then there exist $C=C_{\ep,r} ,K>0$ with $K\in \bbz$ such that for each $t\in[0,T]$, we have
\begin{equation*}
\nor{\Ome_\wta^\ep [\Ome_\wta^{-\ep}, D_0 ] f}{L^2} \le C \jb{\nor{\wta}{X}}^K \left( \nor{\wtf_{01}}{L^r} + \nor{\wtf_{02}}{L^r} \right) 
\nor{f}{L^2},
\end{equation*}
where $\wtf_{0k} := \rd_0 \wta_k - \rd_k A_0 = \Im (\psi_0 \ovl{\psi_k}) -\frac{1}{2} b \rd_t u_k$.
\end{lem}

\begin{proof}
All the proof can be found in \cite{W}, while the only difference here is that the norm of $A$ is measured by $X$, instead of $H^1$ in \cite{W}. 
Therefore, it suffices to show the following inequality:
\begin{equation*}
\nor{Vf}{L^2} \le \ep \nor{(1-\Del) f }{L^2} + C \ep^{-1} \jb{\nor{A}{X}}^4 \nor{f}{L^2},
\end{equation*}
where $V$ is the function defined by $1-\Del_\wta = 1-\Del +V$, or explicitly 
$V= \sum_{k=1}^2 \left(-2i\wta_k \rd_k - i(\rd_k \wta_k) + \wta_k^2 \right)$. However, using the Gagliardo-Nirenberg inequality, we have
\begin{equation*}
\begin{aligned}
&\nor{Vf}{L^2} \\
&\le 
 C\sum_{k=1}^2\nor{\wta_k}{L^4} \nor{\rd_k f}{L^4} 
+C\sum_{k=1}^2\nor{\rd_k \wta_k}{L^2} \nor{f}{L^\infty}
+C (\sum_{k=1}^2\nor{\wta_k}{L^4} )^2 \nor{f}{L^\infty}\\
&\le C \nor{\wta}{X} \nor{f}{L^2}^{1/4} \nor{\Del f}{L^2}^{3/4} 
+ \jb{\nor{\wta}{X}}^2 \nor{f}{L^2}^{1/2} \nor{\Del f}{L^2}^{1/2}\\
&\le \ep \nor{(1-\Del) f }{L^2} + C\ep^{-3} \jb{\nor{\wta}{X}}^4\nor{f}{L^2},
\end{aligned}
\end{equation*}
which is the conclusion. \qed
\end{proof}

\textit{Proof of Proposition \ref{Pa1}}. 
For simplicity of notation, let $C_\wta$ stand for $C \jb{\nor{\wta}{L^\infty_T L^2_x}}^K$ 
for some constant $C, K>0$ with $K\in\bbz$ which are independent of $u$ and $\psi$. \par
Let $\al\in \bbn^2$ with $|\al|\le L$. 
Operating $\Ome_\wta^\ep \rd^\al \psi_m$ on both sides of (\ref{ab1}), 
we obtain 
\begin{equation}\label{ab2}
D_0 \Ome_\wta^\ep \rd^\al \psi_m 
= -i \Del_\wta \psi_m + \sum_{\nu=1}^5 R_\nu ,
\end{equation}
where 
%
\begin{equation*}
R_1:= [D_0, \Ome_\wta^\ep] \rd^\al \psi_m,
\end{equation*}
\begin{equation*}
R_2:= -i \sum_{\beta+ \gam =\al, \beta \neq 0} 
\frac{\al !}{\beta ! \gam !}\, \Ome_\wta^\ep (\rd^\beta A_0\cdot \rd^\gam \psi_m) ,
\end{equation*}
\begin{equation*}
R_3:= 2 \sum_{\beta+ \gam =\al, \beta \neq 0} \sum_{l=1}^2 
\frac{\al !}{\beta ! \gam !}\, \Ome_\wta^\ep (\rd^\beta \wta_l\cdot \nab\rd^\gam \psi_m) ,
\end{equation*}
\begin{equation*}
R_4:= i \sum_{\beta+ \gam =\al, \beta \neq 0} 
\frac{\al !}{\beta ! \gam !}\, \Ome_\wta^\ep (\rd^\beta (
|\wta|^2 
) \cdot \rd^\gam \psi_m) ,
\end{equation*}
\begin{equation*}
R_5 := \Ome_\wta^\ep \rd^\al \can_m
-\frac{b}{2} \Ome_\wta^\ep \rd^\al \left( (\rd_1 u_1 +\rd_2 u_2) \psi_m \right)
+ \frac{b}{2} \Ome_\wta^\ep \left( (\rd_1 u_1 +\rd_2 u_2) \rd^\al \psi_m \right)
.
\end{equation*}
Then we consider the inner product of (\ref{ab2}) and $\Ome_\wta^\ep \rd^\al \psi_m  $, which yields
\begin{equation}\label{bx1}
\nor{\Ome_\wta^\ep \rd^\al \psi_m}{L^\infty_T L^2_x}
\le \left. \nor{\Ome_\wta^\ep \rd^\al \psi_m}{L^2_x} \right|_{t=0}
+
C\sum_{\nu=1}^5 \int_0^T \nor{R_\nu}{L^2_x} dt.
\end{equation}
Now we claim that 
\begin{equation}\label{ac2}
\sum_{\nu=1}^5 \int_0^T \nor{R_\nu}{L^2_x} dt 
\le C_\wta T \nor{\psi}{H^{L+\ep}} P_L(M_T).
\end{equation}
Let us first examine $R_1$. 
We may assume $\ep>0$, since 
$R_1=0$ when $\ep=0$. 
Applying Lemma \ref{La1} (iv) with $r=2/(1-\del)$ for $\del \in (0,\ep)$, we have
\begin{equation}\label{ab3}
\begin{aligned}
\nor{R_1}{L^2_x} 
&= \nor{\Ome_\wta^\ep [\Ome_\wta^{-\ep} , D_0] \Ome_\wta^\ep \rd^\al \psi_m}{L^2_x} \\
&\le C_\wta (\nor{F_{01}}{L^{2/(1-\del)}_x} + \nor{F_{02}}{L^{2/(1-\del)}_x} ) \nor{\Ome_\wta^\ep \rd^\al \psi_m}{L^2_x}.
\end{aligned}
\end{equation}
For $F_{0k}$ ($k=1,2$), Proposition \ref{EP4} gives
\begin{equation*}
\begin{aligned}
\nor{F_{0k}}{L^{2/(1-\del)}_x} 
&\le C \nor{\psi_0 }{L^{2/(1-\del)}_x}  \left( \nor{\psi_k}{L^\infty_x} + 1 \right)\\
&\le \nor{\psi}{H^{1+\del}} P(\nor{\psi}{H^{1+\del}}).
\end{aligned}
\end{equation*}
Thus (\ref{ab3}) is bounded by
\begin{equation*}
C_\wta P(\nor{\psi}{H^{1+\del}_x}) \nor{\Ome_\wta^\ep \rd^\al \psi_m}{L^2_x}.
\end{equation*}

For $R_2$, we divide the proof into the case when $L=1$ and $L\ge 2$. 
If $L=1$, Proposition \ref{EP4}, Lemma \ref{La1} and the Leibniz rule yield
\begin{equation*}
\begin{aligned}
\nor{R_2}{L^2_x}
\le C_\wta \nor{(\rd^\al A_0) \cdot \psi_m }{H^\ep }
&\le C_\wta \left( \nor{\rd^\al  A_0}{H^\ep} \nor{\psi}{L^\infty} 
+ \nor{\rd^\al  A_0}{L^{2/(1-\ep)}} \nor{\psi}{H^\ep_{2/\ep}} \right)\\
&\le C_\wta \nor{\psi}{H^{1+\ep}} P( \nor{\psi}{H^{1+\del}} ).
\end{aligned}
\end{equation*}
When $L\ge 2$, we have
\begin{equation*}
\begin{aligned}
\nor{R_2}{L^2_x}
&\le C_\wta \left( \nor{\rd^\al  A_0}{H^\ep} \nor{\psi}{L^\infty_x} 
+ \nor{\rd^\al  A_0}{L^{2/(1-\ep)}_x} \nor{\psi}{H^\ep_{2/\ep}} \right)\\
&\hspace{20pt} + C_\wta \sum_{\beta+\gam=\al, \beta\neq 0,\al} 
\left( 
\nor{\rd^\beta A_0}{H^\ep_4} \nor{\rd^\gam \psi}{L^4_x} 
+ \nor{\rd^\beta A_0}{L^4_x} \nor{\rd^\gam \psi}{H^\ep_4}
\right)
\\
&\le C_\wta \nor{\psi}{H^{L+\ep}} P_L (M_T).
\end{aligned}
\end{equation*}
Next, when $L= 1$, $R_3$ is estimated as follows:
\begin{equation*}
\begin{aligned}
\nor{R_3}{L^2_x} &\le 
C_\wta 
\sum_{l=1}^2 
\nor{\rd^\al \wta_l \cdot \nab \psi_m }{H^\ep}\\
& \le 
C_\wta \left( \nor{\rd^\al \wta}{H^{\ep}_{2/\ep}} \nor{\nab \psi}{L^{2/(1-\ep)}_x}
+
\nor{\rd^\al \wta}{L^{\infty}_x} \nor{\nab \psi}{H^\ep}
\right)\\
&\le C_\wta \nor{\psi}{H^{1+\ep}} P_L (\nor{\psi}{H^{1+\del}}).
\end{aligned}
\end{equation*}
where we used Proposition \ref{EP2} (iii) in the last inequality. 
If $L\ge 2$, we have
\begin{equation*}
\begin{aligned}
\nor{R_3}{L^2_x} &\le 
C_\wta \sum_{\beta+\gam =\al, |\beta|=1} 
\left( 
\nor{\rd^\beta \wta}{H^\ep_{2/\ep}} \nor{\nab \rd^\gam \psi}{L^{2/(1-\ep)}}
+ \nor{\rd^\beta \wta}{L^\infty} \nor{\nab \rd^\gam \psi}{H^\ep} \right)\\
&\hspace{20pt} + C_\wta 
\sum_{\beta+\gam =\al, |\beta|\ge 2} 
\left( \nor{\rd^\beta \wta}{H^\ep_4} \nor{\nab\rd^\gam \psi}{L^4}
+ \nor{\rd^\beta \wta}{L^4} \nor{\nab \rd^\gam \psi}{H^\ep_4} \right)\\
& \le C_\wta \nor{\psi}{H^{L+\ep}} P_L(M_T) .
\end{aligned}
\end{equation*}
For $R_4$, we obtain
\begin{equation}\label{ac1}
\begin{aligned}
\nor{R_4}{L^2_x}
&\le 
C_\wta \sum_{\beta+\gam=\al, \beta\neq 0} \nor{\rd^\beta (|\wta|^2) \rd^\gam \psi}{H^\ep}\\
&\le C_\wta \left( \nor{\rd^\al (|\wta|^2)}{H^\ep} \nor{\psi}{L^\infty} 
+\nor{\rd^\al (|\wta|^2)}{L^{2/(1-\ep)}} \nor{\psi}{H^{\ep}_{2/\ep}} \right)\\
&\hspace{20pt} + C_\wta \sum_{\beta+\gam=\al, \beta\neq 0,\al} 
\left( \nor{\rd^\beta (|\wta|^2)}{L^4} \nor{\rd^\gam \psi}{L^4} 
+\nor{\rd^\beta (|\wta|^2)}{L^\infty} \nor{\rd^\gam \psi}{H^{\ep}} \right) .
\end{aligned}
\end{equation}
Here we have
\begin{equation*}
\nor{\rd^\al (|\wta|^2)}{H^\ep}
\le 
\sum_{\al_1+\al_2 =\al} 
\sum_{k=1,2} \nor{\rd^{\al_1} \wta_k }{H^\ep_4} \nor{\rd^{\al_2} \wta_k }{L^4}
\le C M_T^4 .
\end{equation*}
Thus (\ref{ac1}) is bounded by
\begin{equation*}
C_\wta \nor{\psi}{H^{L+\ep}} P_L(M_T).
\end{equation*}
For $R_5$, we first estimate $\nor{\Ome_\wta^\ep \rd^\al \psi_m}{L^2_x}$. 
To this end, we observe that $\can_m$ is a linear combination of the followings:
\begin{equation*}
(\rd_j u_k) \psi_m, \qquad u_k^2 \psi_m, \qquad \im(\psi_m \ovl{\psi_k}) \psi_k. 
\qquad (j,k=1,2)
\end{equation*}
For the first one, we have
\begin{equation*}
\begin{aligned}
\nor{\Ome_\wta^\ep \rd^\al ((\rd_j u_k)\psi_m)}{L^2_x}
&\le C_\wta \nor{ (\rd_j u_k)\psi_m}{H^{L+\ep}}\\
&\le C_\wta  \left( \nor{ \rd_j u_k}{H^{L+\ep}} \nor{\psi}{L^\infty_x} 
+ \nor{\rd_j u_k}{L^\infty_x} \nor{\psi}{H^{L+\ep}} \right)\\
&\le C_\wta \nor{\psi}{H^{L+\ep}} P_L (M_T).
\end{aligned}
\end{equation*}
For the third one, the Leibniz rule yields
\begin{equation*}
\nor{\Ome_\wta^\ep \rd^\al (\im(\psi_m \ovl{\psi_k}) \psi_k)}{L^2_x}
\le C_\wta  \nor{\im(\psi_m \ovl{\psi_k}) \psi_k}{H^{L+\ep}_x} \le C_\wta \nor{\psi}{H^{L+\ep}} M_T^2.
\end{equation*}
For the second one, we have
\begin{equation*}
\begin{aligned}
\nor{\Ome_\wta^\ep \rd^\al (u_k^2 \psi_m )}{L^2_x}
&\le C_\wta \left( \nor{u_k^2 \rd^\al \psi_m}{H^\ep} 
+\sum_{\al_1+\al_2=\al, \al_1 \neq 0} \nor{ \rd^{\al_1} u_k \cdot u_k \cdot \rd^{\al_2} \psi_m }{H^{\ep}} \right. \\
&\hspace{30pt} \left. +\sum_{\al_1+\al_2+\al_3=\al, \al_1, \al_2 \neq 0} \nor{ \rd^{\al_1} u_k \cdot \rd^{\al_2} u_k \cdot \rd^{\al_3} \psi_m }{H^{\ep}}\right) \\
& =: R_{51} + R_{52} + R_{53},
\end{aligned}
\end{equation*}
Each of these is estimated as follows:
\begin{equation*}
\begin{aligned}
R_{51} &\le C_\wta \left(
\nor{u_k^2 \psi_m}{L^2}^2 + \nor{|\nab|^\ep (u_k^2\psi_m)}{L^2}
\right) \\
&\le C_\wta \left( \nor{\psi}{H^L} + \nor{|\nab|^\ep u_k }{L^{2/\ep}} \nor{u_k}{L^\infty} 
\nor{\rd^\al \psi_m}{L^{2/(1-\ep)}}  + \nor{u_k}{L^\infty}^2 \nor{|\nab|^\ep \rd^\al \psi_m}{L^2} \right) \\
&\le C_\wta \nor{\psi}{H^{L+\ep}} P_L (M_T),
\end{aligned}
\end{equation*}
\begin{equation*}
\begin{aligned}
& R_{52} \le C_\wta  \sum_{\al_1+\al_2=\al, \al_1 \neq 0} \left(
\nor{\rd^{\al_1} u_k \cdot u_k \cdot \rd^{\al_2} \psi_m }{L^2_x} +
\nor{|\nab|^\ep (\rd^{\al_1} u_k \cdot u_k \cdot \rd^{\al_2} \psi_m )}{L^2_x}
\right)\\
&\le 
C_\wta  \sum_{\al_1+\al_2=\al, \al_1 \neq 0} (
\nor{\rd^{\al_1} u_k}{L^4} \nor{\rd^{\al_2} \psi}{L^4} + \nor{|\nab|^\ep \rd^{\al_1} u_k }{L^{2/\ep}} \nor{u_k}{L^\infty} \nor{\rd^{\al_2} \psi }{L^{2/(1-\ep)}}\\
&\hspace{100pt} + \nor{\rd^{\al_1} u_k}{L^\infty} \nor{\nab|^\ep u}{L^{2/\ep}} \nor{\rd^{\al_2} \psi}{L^{2/(1-\ep)}} \\
&\hspace{100pt} + \nor{\rd^{\al_1} u_k}{L^\infty} \nor{u_k}{L^\infty} \nor{|\nab|^\ep \rd^{\al_2} \psi_m}{L^2}
)\\
&\le C_\wta  \nor{\psi}{H^{L+\ep}} P_L (M_T),
\end{aligned}
\end{equation*}
\begin{equation*}
R_{53} \le C_\wta \nor{\psi}{H^{L+\ep}} P_L(M_T).
\end{equation*}
The other terms in $R_5$ can be similarly estimated to the above. 
Hence (\ref{ac2}) is proved. \bigskip\par
Next we estimate $\nor{\wta}{L^\infty_T X}$. 
Using Lemma \ref{EL1} and (\ref{aa15}), we have
\begin{equation*}
\begin{aligned}
\nor{\wta}{L^\infty_T X} &\le C\left( \nor{A}{L^\infty_TL^4_x} + 
\nor{\nab A}{L^\infty_TL^2_x} + \nor{u-\vec k}{L^\infty_T L^4_x} + 
\nor{\nab u}{L^\infty_T L^2_x}
\right)\\
&\le C\left( \nor{u-\vec k}{L^\infty_T L^2_x} + \nor{\psi}{L^\infty_TL^2_x} +\nor{\psi}{L^\infty_TL^2_x}^2 + \nor{\psi}{L^\infty_TL^4_x}^2   
\right).\\
\end{aligned}
\end{equation*}
Here we observe that the following inequalities hold true:
\begin{equation}\label{ay5}
\nor{\psi}{L^\infty_T L^2_x} \le C \left( 
\nor{\nab u_0}{L^2_x} + T (\nor{\psi}{L^\infty_T L^2_x} + \nor{\psi}{L^\infty_T L^2_x}^3)
\right),
\end{equation}
\begin{equation}\label{ay6}
\nor{\psi}{L^\infty_T L^4_x} \le C \left( 
\nor{\nab u_0}{L^4_x} + T (M_T + M_T^{3/2})
\right),
\end{equation}
Indeed, (\ref{ay5}) immediately follows from (\ref{ab1}), and (\ref{ay6}) follows from
\begin{equation*}
\begin{aligned}
\rd_t \nor{\psi}{L^4_x}^4 
&= 4 \re \int_{\R^2} |\psi_m|^2 \ovl{\psi_m} D_0 \psi_m  dx \\
& = 4\re \sum_{k=1}^2 \int_{\R^2} \rd_k (|\psi_m|^2 ) \ovl{\psi_m} \widetilde{D}_k\psi_m dx
+ 4\re \int_{\R^2} |\psi_m|^2 \ovl{\psi_m} \can_m dx.
\end{aligned}
\end{equation*}
Combining (\ref{aa15}), (\ref{ay5}) and (\ref{ay6}), we have
\begin{equation}\label{bx2}
\nor{\wta}{L^\infty_T X} \le C (\nor{u_0-\vec k}{H^{3/2}} )
+ C(T^{1/2} +T )\left( M_T +M_T^3
\right) .
\end{equation}
Applying (\ref{ac2}) and (\ref{bx2}) to (\ref{bx1}), we obtain the desired estimate.\qed

%
%
%
%

\section{Continuity}\label{Sc6}
In this section, we prove Theorem \ref{T2} (d). 
Our strategy is based on \cite{BIKT}. Namely, we first take a smooth homotopy map connecting the two solutions. Then we can consider the derivative with respect to the new parameter, 
thus can define the associated differentiated field. 
We will show that this quantity has quantitative equivalence to the difference of two solutions, and hence we can reduce the problem to its estimate. \bigskip\par
Letting $\Ome := \R^3\backslash\{ 0\}$, 
we begin by defining $\Pi :\Ome \to \bbs^2$ by 
$\Pi (y) := \frac{y}{|y|}$. 
Then $\Pi$ is smooth, and 
$\Pi (p) = p$ for all $p\in\bbs^2$.\par
Let us first prove (\ref{ca1}). 
Let $u^{(0)}_0, u^{(1)}_0 \in \vec k + H^s$. 
Without loss of generality, we may assume that 
$(1-h) u^{(0)}_0 + h u^{(1)}_0 \in \Ome$ 
by taking $\nor{u^{(1)}_0 -u^{(0)}_0  }{H^s} \ll 1$. 
Furthermore, we may also assume that $u^{(0)}, u^{(1)} \in \vec k + H^\infty$, 
since the general case follows from the continuity of solution map shown later.\par
For $h \in [0,1]$, we define 
\begin{equation*}
u^{(h)}_0 := \Pi \circ \left( (1-h) u^{(0)}_0 + h u^{(1)}_0 \right).
\end{equation*}
Note that this definition is consistent when $h=0,1$. 
Let $u^{(h)}\in C_T (\vec k +H^\infty)$ be the solution with $u^{(h)}|_{t=0}= u^{(h)}_0$. 
It follows that 
$u^{(h)}$ is defined in $t\in [0,T]$ with $T>0$ independent of $h$, 
which is clear from the consequence of Section \ref{Sc5}. 
Then the following estimate holds. 

\begin{prop}\label{Pad1}
For $\sig\ge 0$, there exists a polynomial $P= P_{\lfloor \sig \rfloor}$ such that the following holds for all $t\in [0,T]$ and $h\in [0,1]$.
\begin{equation}\label{ad1}
\nor{\rd_h u^{(h)}}{H^\sig} \le \nor{u^{(1)} - u^{(0)}}{H^\sig} 
P\left( \nor{u^{\max}-\vec k}{H^{\max \{ \sig , 1\} }} \right) ,
\end{equation}
where by $u^{\max}$ we follow the same convention as in Section \ref{Sc3}. 
\end{prop}
\begin{rmk}
The differentiability with respect to $h$ is justified 
in the following way. 
Let $h_0, h_1 \in [0,1]$. 
Since both of $u^{(h_0)}$, $u^{(h_1)}$ satisfy (\ref{a1}), a similar argument to (\ref{by1}) 
yields
\begin{equation*}
\nor{u^{(h_0)} - u^{(h_1)}}{L^\infty_{t,x}} \le 
C\nor{u^{(h_0)} - u^{(h_1)}}{L^\infty_T H^d} \le
C\nor{u^{(h_0)}_0 - u^{(h_1)}_0}{H^d},
\end{equation*}
where $C=C(\nor{u^{(h_0)}_0 - \vec k}{H^N}, \nor{u^{(h_1)}_0 - \vec k}{H^N})$ 
for sufficiently large $N$. Hence the claim follows from absolute continuity of initial data with respect to $h$.
\end{rmk}
\begin{proof}
When $\sig\in\bbz$, (\ref{ad1}) follows from the chain rule. 
If $\sig \in (0,1)$, we have
\begin{equation*}
\begin{aligned}
&\nor{\rd_h u^{(h)}}{\dot{H}^\sig} \\
&\le 
C \nor{\left( (\nab \Pi) \circ ((1-h)u^{(0)} + hu^{(1)} )\right) \cdot 
 (u^{(1)} - u^{(0)}) }{H^\sig} \\
&\le 
C_\sig \nor{|\nab|^\sig \left(  \left(\nab \Pi \circ 
((1-h) u^{(0)} + hu^{(1)}) - \nab\Pi (\vec k)\right) 
\cdot (u^{(1)}-u^{(0)}) \right)}{L^2_x} \\
&\hspace{190pt} +
C_\sig \nor{u^{(1)} -u^{(0)} }{H^\sig}\\
&\le C_\sig \nor{ |\nab|^\sig \left( 
\nab\Pi \circ ((1-h) u^{(0)} + hu^{(1)}) - \nab\Pi (\vec k)
\right) }{L^{2/\sig}_x}
\nor{u^{(1)}-u^{(0)} }{L^{2/(1-\sig)}_x}\\
&\hspace{15pt} +
C_\sig \nor{\nab\Pi \circ ((1-h) u^{(0)} + hu^{(1)}) - \nab\Pi (\vec k)}{L^\infty_x}
\nor{\nab^\sig (u^{(1)}- u^{(0)}) }{L^2_x}\\
&\hspace{15pt} + C_\sig \nor{u^{(1)}-u^{(0)}}{H^\sig}\\
&\le C_\sig \left( \nor{\nab \left( \nab\Pi \circ ((1-h) u^{(0)} + hu^{(1)}) - \nab\Pi (\vec k)
\right)}{L^2}  + 1\right) \nor{u^{(1)}- u^{(0)}}{H^\sig}\\
&\le C_\sig \nor{u^{(1)}- u^{(0)} }{H^\sig} P (\nor{u^{\max} -\vec k}{H^{\max \{ \sig ,1 \} }}).
\end{aligned}
\end{equation*}
The other case can be proved similarly. \qed
\end{proof}

Next, let $v,w$ be an orthonormal frame of $T_{u^{(h)}} \bbs^2$ with 
\begin{equation*}
v-\vec k_1, w-\vec k_2\in C_{t,h} (L^1 + L^r)\cap C_{t,h} \nab H^\infty,
\end{equation*}
and with
\begin{equation*}
\rd_1 A_1 + \rd_2 A_2 =0 \text{ for all } (x,t,h)\in \R^2\times [0,T]\times [0,1], 
\end{equation*}
where $A_m= A_m(x,t,h) := \rd_j v^{(h)}\cdot w^{(h)}$ for $m=0,1,2$. 
(Such frame can be constructed in the same manner as Section \ref{Sc4}.) 
We also define
$\psi_m$ and $D_m$ for $m=0,1,2$ in the same way as Section \ref{Sc4}. 
Now we assign $m=3$ to the $h$-variable, and define $\psi_3$, $A_3$ in the same manner. 
Then it follows that (\ref{g1}), (\ref{g2}), (\ref{g4}) and (\ref{ab1}) hold even when $m=3$. 
Especially, the same argument as in Sections \ref{Sc4} and \ref{Sc5} 
yields the following estimates:
\begin{equation*}
\nor{\rd_h u^{(h)}}{H^{s-1}} \le \nor{\psi_3}{H^{s-1}} P_{\lfloor s-1 \rfloor} (\nor{\psi}{H^{s-1}}) \quad \text{ for } t\in [0,T],
\end{equation*}
\begin{equation*}
\nor{\psi_3}{H^{s-1}} \le \nor{\rd_h u^{(h)}}{H^{s-1}} P_{\lfloor s-1 \rfloor} (\nor{\nab u^{(h)}}{H^{s-1}}) \quad \text{ for } t\in [0,T],
\end{equation*}
\begin{equation*}
\nor{\psi_3}{L^\infty_TH^{s-1}} \le C \nor{\rd_h u^{(h)}_0}{H^{s-1}} + C T (1+T)^K \nor{\psi_3}{L^\infty_T H^{s-1}}.
\end{equation*}
where $C=C(\nor{u^{\max}_0 - \vec k}{H^{s-1}})>0$, $K>0$ are some constants. Thus we have
\begin{equation*}
\nor{\psi_3}{L^\infty_TH^{s-1}} \le C(\nor{u^{\max}_0 - \vec k}{H^{s}}) \nor{\rd_h u^{(h)}_0}{H^{s-1}}
\end{equation*}
for some $T=T(\nor{u^{\max}_0 - \vec k}{H^{s}})$. Consequently, Proposition \ref{Pad1} yields
\begin{equation*}
\begin{aligned}
\nor{u^{(1)}-u^{(0)} }{L^\infty_T H^{s-1}}
&\le \int_0^1 \nor{\rd_h u^{(h)} }{L^\infty_T H^{s-1}}\\
&\le \int_0^1 \nor{\psi_3}{L^\infty_T H^{s-1}} P(\nor{\psi}{L^\infty_T H^{s-1}})\\
&\le \int_0^1 \nor{\rd_h u^{(h)}_0}{ H^{s-1}} C(\nor{u^{\max}-\vec k}{L^\infty_T H^{s}})\\
&\le C(\nor{u^{\max}_0-\vec k}{H^{s}}) \nor{u^{(1)}_0-u^{(0)}_0}{H^{s-1}},
\end{aligned}
\end{equation*}
which gives (\ref{ca1}).\bigskip\par
We next prove the continuity part. The proof here is based on the typical argument invented in \cite{BS}, while a detailed exposition of this method is available in \cite{ET}. 
Let $s>2$ and 
suppose that $\{ u_0^{(n)}\}_{n\in \bbn}$ is a sequence in $\vec k +H^s$ satisfying $u_0^{(n)} -u_0\to 0$ in $H^s$. 
Let $\vph\in \mathcal{S} (\R^2)$ be a nonnegative function with $\nor{\vph}{L^1}=1$ satisfying $\caf [\vph] = 1$ in a neighborhood of the origin. 
Then we define
\begin{equation*}
u_{0,\eta} := \Pi \circ \left( \vph_\eta * \left( u_0-\vec k \right) + \vec k\right) ,
\end{equation*}
where $\vph_\eta := \eta^{-2} \vph (\eta^{-1} \cdot)$, and also define $u_{0,\eta}^{(n)}$ in the same way. 
Let $u$, $u_\eta$, and $u_{\eta}^{(n)}$ be unique solutions to (\ref{a1}) corresponding to the initial data $u_0$, $u_{0,\eta}$, and $u_{0,\eta}^{(n)}$ respectively. 
Then we have the following properties:
\begin{prop}\label{Px1}
The followings hold true.\\
(i) There exist $\eta_0> 0$, $n_0\in \bbn$ such that $u_{0,\eta}$, $u^{(n)}_{0,\eta}$ are well-defined for all $0< \eta < \eta_0$ and $n\ge n_0$.\\
(ii) $\sup_{n\ge n_0, \eta\ge\eta_0} \nor{u_{0,\eta}^{(n)} -\vec k}{H^{s}} <\infty$. 
Thus the maximal existence time of $u_\eta$ and $u_{\eta}^{(n)}$ are bounded from below uniformly in $\eta$ and $n$. \\
(iii) $\nor{u_{0,\eta} -\vec k}{H^{s+1}} + \sup_{n\ge n_0} \nor{u_{0,\eta}^{(n)} -\vec k}{H^{s+1}}  \le C \eta^{-1}$.\\
(iv) There exists $T>0$, independent of $\eta$, such that 
\begin{equation*}
\nor{u_\eta - u}{L^\infty_T H^s} + \sup_{n\ge n_0} \nor{u_{\eta}^{(n)} - u^{(n)}}{L^\infty_T H^s} \to 0 \text{ as } \eta\to 0.
\end{equation*}
(v) There exists $T>0$, independent of $\eta$, such that 
\begin{equation*}
\nor{u_\eta - u_\eta^{(n)} }{L^\infty_T H^s} \le C(\eta) \nor{u_{0} -u_0^{(n)}}{H^s}.
\end{equation*}
\end{prop}

\begin{proof}
(i) Since $\Pi$ is defined in $\R^3\backslash \{0\}$, it suffices to find $\eta_0>0 ,n_0\in\bbn$ satisfying
\begin{equation}\label{x1}
\nor{u_0^{(n)}- \vec k - \vph_\eta * (u_0^{(n)}- \vec k)}{L^\infty} < 10^{-2}
\end{equation}
for all $0< \eta < \eta_0$ and $n\ge n_0$. 
We observe that for $\Phi \in \cas (\R^2)$, 
the left hand side is bounded by
\begin{equation*}
C (1+ \nor{\vph}{L^1} ) \left( \nor{u_0 - u_0^{(n)} }{H^s} + \nor{u_0-\vec k - \Phi}{H^s} \right) 
+ \nor{\Phi - \vph_\eta *\Phi }{L^\infty}
\end{equation*}
by the Young inequality and the Sobolev embeddings. 
Thus the conclusion follows since $\cas (\R^2)$ is dense in $H^s(\R^2)$. \bigskip\par
(ii), (iii) By (\ref{x1}), the mean value theorem implies
\begin{equation*}
\begin{aligned}
\nor{u_{0,\eta}^{(n)} -\vec k }{L^2}
&\le \nor{\Pi ( \vph_\eta * (u_0^{(n)}- \vec k)  +\vec k) - 
\Pi ( u_0^{(n)}) }{L^2}
+
\nor{ u_0^{(n)} - \vec k}{L^2}\\
&\le 
C \nor{ \vph_\eta * (u_0^{(n)}- \vec k) - (u_0^{(n)}- \vec k) }{L^2}
+
\nor{ u_0^{(n)}  -\vec k}{L^2}\\
&\le C (1+\nor{\vph}{L^1}) \nor{ u_0^{(n)} - \vec k }{L^2},
\end{aligned}
\end{equation*}
which is bounded uniformly in $\eta$ and $n$. 
For $\sig\ge 1$, we have
\begin{equation*}
\begin{aligned}
&\nor{\nab^\sig (u_{0,\eta}^{(n)} -\vec k)}{L^2_x} \\
&\le C \sum_{j=1}^2 \nor{ \left(
\nab\Pi \circ (\vph_\eta * (u^{(n)} -\vec k) + \vec k)
 \right) \cdot \rd_j (\vph_\eta * (u^{(n)} -\vec k) )}{H^{\sig-1}}\\
&\le \nor{\nab \left( \vph_\eta * (u_0^{(n)} - \vec k) \right)}{H^{\sig -1}}
 P(\nor{\vph_\eta * (u^{(n)} -\vec k) }{H^{\sig -1}}),
\end{aligned}
\end{equation*}
where the last inequality follows from the same argument as in the proof of Proposition \ref{Pad1}, which leads to the conclusion.\bigskip\par
%
%
(iv) We first observe the convergence of the first term. 
The argument in Step 1 in the proof of Theorem \ref{T2} in Section \ref{Sc5} yields 
\begin{equation*}
\nor{u_\eta - \vec k}{L^\infty_T H^{s+1}} \le C(\nor{u_{0,\eta} -\vec k}{H^s}) 
\nor{u_{0,\eta} -\vec k}{H^{s+1}}
\end{equation*}
for $T=T(\nor{u_0-\vec k }{H^{s}})$. 
Especially, (iii) implies 
\begin{equation*}
\nor{u_\eta -\vec k}{L^\infty_T H^{s+1}} \le C \eta^{-1}.
\end{equation*}
Thus by interpolation, it suffices to show 
\begin{equation*}
\nor{u_\eta - u_{\eta'}}{H^1} = o(\eta^{s-1}) \quad \text{ for } 0<\eta'< \eta. 
\end{equation*}
We first note that the following $H^1$-difference estimate is true:
\begin{equation*}
\nor{u_\eta - u_{\eta'}}{L^\infty_T H^1} \le C \nor{u_{0,\eta} - u_{0,\eta'}}{H^1}.
\end{equation*}
Indeed, since the embedding $\nab^{-1} L^\infty \supset H^s$ holds for $s>2$, 
we can use in Section \ref{Sc3} the energy method instead of Yudovich argument, 
which yields the estimate of the form $G(T)\le CG(0)$ 
(see also \cite{M}). Hence we have
\begin{equation}\label{ax2}
\begin{aligned}
\eta^{2-2s} \nor{u_\eta - u_{\eta'}}{L^\infty_T H^1}^2 
&\le C \eta^{2-2s} \nor{u_{0,\eta} - u_{0,\eta'} }{H^1}^2\\
&\le C \eta^{2-2s} \nor{ (\vph_\eta - \vph_{\eta'} ) * (u_0 -\vec k) }{H^1}^2\\
&\le C \eta^{2-2s} \int_{\R^2} | \caf [\vph_\eta] (\xi) - \caf [\vph_{\eta'}] (\xi) |^2
|\caf [u_0 -\vec k] (\xi)|^2 \jb{\xi}^2 d\xi.
\end{aligned}
\end{equation}
By the definition of $\vph$, the Taylor theorem and interpolation yield
\begin{equation*}
|\caf[\vph_\eta] (\xi) -1 |\le C \eta^{s-1} |\xi|^{s-1} 
\left( 
\sup_{|\xi'|\le \eta |\xi|} |\rd^{s} \caf \vph(\xi')| + \sup_{|\xi'|\le \eta |\xi|} |\rd^{s+1} \caf \vph(\xi')| 
\right). 
\end{equation*}
Thus (\ref{ax2}) is bounded by
\begin{equation*}
C\int_{\R^2} \left( 
\sup_{|\xi'|\le \eta |\xi|} |\rd^{s} \caf \vph(\xi')| + \sup_{|\xi'|\le \eta |\xi|} |\rd^{s+1} \caf \vph(\xi')| 
\right) |\caf [u_0-\vec k](\xi)|^2 \jb{\xi}^{s} d\xi \to 0
\end{equation*}
as $\eta\to 0$ by the dominant convergence theorem. \par
The convergence of the second term can be similarly shown since
\begin{equation*}
\begin{aligned}
&\int_{\R^2} 
\sup_{|\xi'|\le \eta |\xi|} |\rd^{s} \caf \vph(\xi')| + \sup_{|\xi'|\le \eta |\xi|} |\rd^{s+1} \caf \vph(\xi')| 
 |\caf [u^{(n)}_0 - u_0] (\xi) |^2 \jb{\xi}^s d\xi\\
&\le C \nor{u^{(n)}_0 - u_0}{H^s} \to 0
\end{aligned}
\end{equation*}
as $n\to\infty$ uniformly in $\eta$.\bigskip\par 
(v) Applying (\ref{ca1}), we have
\begin{equation*}
\begin{aligned}
\nor{u_\eta - u_\eta^{(n)}}{L^\infty_T H^s} 
&\le C(\nor{u_{0,\eta} -\vec k}{H^{s+1}},
\nor{u_{0,\eta}^{(n)} -\vec k}{H^{s+1}}
)
\nor{u_{0,\eta} - u_{0,\eta}^{(n)}}{H^s}\\
&\le C(\eta) \nor{u_0 - u_0^{(n)}}{H^s},
\end{aligned}
\end{equation*}
which completes the proof.\qed
\end{proof}

We finally show that Proposition \ref{Px1} implies the continuity. 
Let $\ep>0$ be arbitrary number. Then there exists $T>0$ such that for $\eta \ge \eta_0$ and $n> n_0$, 
we have
\begin{equation*}
\begin{aligned}
\nor{u - u^{(n)}}{L^\infty_T H^s} &\le \nor{u-u_\eta}{L^\infty_T H^s} 
+ \nor{u_\eta - u_\eta^{(n)}}{L^\infty_T H^s} +\nor{u_\eta^{(n)}-u^{(n)}}{L^\infty_TH^s} \\
&\le 
\nor{u-u_\eta}{L^\infty_TH^s} 
+ C(\eta) \nor{u_0 - u_0^{(n)} }{H^s} + \sup_{n\ge n_0} \nor{u^{(n)}-u_\eta^{(n)}}{L^\infty_TH^s}.
\end{aligned}
\end{equation*}
We fix sufficiently small $\eta$, then (iv) gives
\begin{equation*}
\nor{u - u^{(n)}}{L^\infty_TH^s} \le 2\ep  
+ C(\eta) \nor{u_0 - u_0^{(n)} }{H^s}.
\end{equation*}
Therefore, taking $n\to \infty$, we have $\lim_{n\to \infty} \nor{u - u^{(n)}}{L^\infty_T H^s} =0$ since $\ep>0$ is arbitrary, which finishes the proof.\bigskip\par


\textit{Acknowledgments}. 
The author would like to thank Yoshio Tsutsumi for the supervision and giving him a lot of useful suggestions. 
The author also wishes to thank Stephen Gustafson for encouraging him to study the problem in the present paper and giving him helpful comments. 
The author was supported by Grand-in-Aid for JSPS Fellows 
18J21037. 

\noindent{}Ikkei Shimizu\par 
\noindent{}Department of Mathematics, Graduate School of Science, Kyoto University, 
Oiwakecho, Kitashirakawa, Sakyoku, Kyoto, 
606-8502, Japan


\begin{thebibliography}{99}
	\bibitem{BIK} Bejenaru, I., Ionescu, A.D., Kenig, C.E.: Global existence and uniqueness of Schr\"odinger maps in dimensions $d\ge 4$. Adv. Math. \textbf{215}(1), 263-291 (2007)
	\bibitem{BIKT} Bejenaru, I., Ionescu, A.D., Kenig, C.E., Tataru, D.: Global Schr\"odinger maps in dimensions $d\ge 2$: Small data in the critical Sobolev spaces. 
Ann. of Math. (2) \textbf{173}(3), 1443-1506 (2011)
	\bibitem{BIKT2} Bejenaru, I., Ionescu, A.D., Kenig, C.E., Tataru, D.: Equivariant Schr\"{o}dinger maps in two spatial dimensions. 
Duke Math. J. \textbf{162}(11), 1967-2025 (2013)
	\bibitem{BT} Bejenaru, I., Tataru, D.: Near soliton evolution for equivariant Schr\"odinger maps 
in two spatial dimensions. 
Amer. Math. Soc., Providence, (2014)
	\bibitem{BS} Bona, J.L., Smith, R.: The initial-value problem for the Korteweg-de Vries equation. Philos. Trans. Roy. Soc. London Ser. A \textbf{278}(1287), 555-601 (1975)
	\bibitem{CSU} Chang, N.H., Shatah, J., Uhlenbeck, K.: Schr\"{o}dinger maps, 
Comm. Pure Appl. Math. \textbf{53}(5), 590-602 (2000)
	\bibitem{dLG} de Laire, A., Gravejat, P.: The Sine-Gordon regime of the Landau-Lifshitz equation with a strong easy-plane anisotropy. Ann. Inst. H. Poincar\'e. 
Anal. Non Lin\'eaire \textbf{35}(7), 1885-1945 (2018)
	\bibitem{DM} D\"oring, L., Melcher, C.: Compactness results for static and dynamic chiral skyrmions near the conformal limit. Calc. Var. Partial Differential Equations \textbf{56}, 60 (2017)
	\bibitem{DS} Dodson, B., Smith, P.: A controlling norm for energy-critical Schr\"odinger maps. Trans. Amer. Math. Soc. \textbf{367}(10), 7193-7220 (2015)
	\bibitem{ET} Erdogan, M.B., Tzirakis, N.: Dispersive Partial 
Differential Equations: Wellposedness and Applications. Cambridge University Press,  Cambridge (2016)
	\bibitem{GKT} Gustafson, S., Kang, K., Tsai, T.-P.: Schr\"odinger flow near harmonic maps. Comm. Pure Appl. Math. \textbf{60}(4), 463-499 (2007)
	\bibitem{GKT2} Gustafson, S., Kang, K., Tsai, T.-P.: 
Asymptotic stability of harmonic maps under the Schr\"odinger flow. Duke Math. J. \textbf{145}(3), 537-583 (2008)
	\bibitem{GK} {\sc S. Gustafson and E. Koo}, Global well-posedness for $2D$ radial \scr maps into the sphere, Preprint, arXiv:1105.5659v1.
	\bibitem{GNT} Gustafson, S., Nakanishi, K., Tsai, T.-P.: Asymptotic stability, concentration, 
and oscillation in harmonic map heat-flow, Landau-Lifshitz, and Schr\"odinger maps on $\R^2$. 
Comm. Math. Phys. \textbf{300}(1), 205-242 (2010)
	\bibitem{IK} Ionescu, A.D., Kenig, C.E.: Low-regularity Schr\"odinger maps, II: 
Global well-posedness in dimensions $d\ge 3$. Comm. Math. Phys. \textbf{271}(2), 523--559 (2007)
	\bibitem{LM} Li, X., Melcher, C.: Stability of axisymmetric chiral skyrmions. J. Funct. Anal. \textbf{275}(10) ,2817-2844 (2018)
	\bibitem{M} McGahagan H.: An approximation scheme for Schr\"odinger maps. Comm. Partial Differential Equations \textbf{32}(1-3), 375-400 (2007)
	\bibitem{Mel} Melcher C.: Chiral skyrmions in the plane, 
Proc. R. Soc. Lond. Ser. A \textbf{470}(2172), 20140394 (2014)
	\bibitem{MRR} Merle, F., Rapha\"el, P., Rodnianski, I.: 
Blowup dynamics for smooth data equivariant solutions to the critical Schr\"odinger map problem. 
Invent. Math. \textbf{193}(2), 249-365 (2013)
	\bibitem{O} Ogawa, T.: A proof of Trudinger's inequality and its application to nonlinear Sch\"odinger equations, Nonlinear Anal. \textbf{14}(9), 765-769 (1990)
	\bibitem{P} Perelman, G.: Blow up dynamics for equivariant critical Schr\"odinger maps. 
Comm. Math. Phys. \textbf{330}(1), 69-105 (2014)
	\bibitem{S2} Shimizu, I.: On uniqueness for Schr\"odinger maps with low regularity large data. Differential Integral Equations \textbf{33}(5-6), 207-222 (2020)
	\bibitem{S} Shimizu, I.: Remarks on local theory for Schr\"odinger maps near  harmonic maps. Kodai Math. J. \textbf{43}(2), 278-324 (2020)
	\bibitem{Sm} Smith, P.: Conditional global regularity of Schr\"odinger maps: 
subthreshold dispersed energy. Anal. PDE \textbf{6}(3), 601-686 (2013)
	\bibitem{SSB} Sulem, P.-L., Sulem, C., Bardos, C.: 
On the continuous limit for a system of classical spins. Comm. Math. Phys. \textbf{107}(3), 431-454 (1986)  
	\bibitem{W} Wada, T.: Smoothing effects for Schr\"odinger equations with electro-magnetic potentials and applications  to the Maxwell-Schr\"odinger equations. 
J. Funct. Anal. \textbf{263}(1), 1-24 (2012)
	\bibitem{Y} Yudovich, V.I.: Non-stationary flows of an ideal incompressible fluid. \u{Z}. Vy\u{c}isl. Mat. I Mat. Fiz. \textbf{3}, 1032-1066 (1963)
\end{thebibliography}
\end{document}